\documentclass[reqno,a4paper]{amsart}

\usepackage{dsfont}
\usepackage{amssymb}
\usepackage{amsmath}
\usepackage{amsfonts}
\usepackage[round]{natbib}
\usepackage[english]{babel}

\usepackage{slashbox}
\usepackage{graphicx}
\usepackage{color}

\numberwithin{equation}{section}
\newtheorem{theo}{Theorem}[section]
\newtheorem{ass}{Assumption}[section]
\newtheorem{lma}{Lemma}[section]

\newtheorem{pro}{Proposition}[section]
\newtheorem{rmrk}{Remark}[section]

\newcommand{\abs}[1]{\lv #1\rv}

\newcommand{\bo}{\beta_{\sste 0}}
\newcommand{\bun}{\beta_{\sste 1}}
\newcommand{\cal}{\mathcal}
\newcommand{\cov}{{\mrm{Cov\,}}}

\newcommand{\dy}{\,{\mrm{d}}y}
\newcommand{\dz}{\,{\mrm{d}}z}
\newcommand{\dt}{\,{\mrm{d}}t}

\newcommand{\dv}{\,{\mrm{d}}v}
\newcommand{\dw}{\,{\mrm{d}}w}

\newcommand{\dn}{\delta_n}
\newcommand{\dsp}{\displaystyle}

\newcommand{\esp}{\mathds{E}}

\newcommand{\fa}{\forall\,}

\newcommand{\h}{\widehat{\beta}_n^{(p)}}

\newcommand{\indi}{\mathds{1}}
\newcommand{\p}{\mathds{P}}
\newcommand{\intd}{\int\!\!\!\int}

\renewcommand{\ll}{\mathds{L}}

\newcommand{\lv}{\left\vert}
\newcommand{\mrm}{\mathrm}
\newcommand{\n}{{\mathds{N}}}

\newcommand{\ra}{r^{\ast}}

\newcommand{\reel}{{\mathds{R}}}

\newcommand{\ro}{r_{\sste 0}}

\newcommand{\rr}{\overline r}

\newcommand{\rv}{\right\vert}
\newcommand{\sste}{\scriptscriptstyle}

\newcommand{\tv}{\xrightarrow}
\newcommand{\var}{{\mrm{Var\,}}}
\newcommand{\ve}{\varepsilon}

\newcommand{\vone}{{\mathbf 1}}
\newcommand{\vp}{\varphi}
\newcommand{\vu}{{\mathbf u}}
\newcommand{\vQ}{{\mathbf Q}}
\newcommand{\wt}{\widetilde}

\def\as{^*}
\def\bN{{\mathbb N}}

\def\T{^{\sste\top}}

\def\tX{\widetilde X}
\def\tXr{\widetilde X_r}

\newcommand{\tta}[2]{\( \begin{matrix} #1 \\[-5pt] {{\scriptstyle (#2)}} \end{matrix} \)}

\numberwithin{figure}{section}
\numberwithin{table}{section}

\title[Global smoothness estimation of a Gaussian process]{Global smoothness estimation of a Gaussian process from regular sequence designs}
\author[D. Blanke]{Delphine Blanke}
\address{Avignon Universit\'{e}, LMA EA2151, 33 rue Louis Pasteur, F-84000 Avignon,~France.}
\email{delphine.blanke@univ-avignon.fr}
\author[C. Vial]{C\'{e}line Vial}
\address{Universit\'{e} de Lyon, CNRS UMR 5208, Polytech Lyon-Universit\'{e} de Lyon 1, Institut Camille Jordan,
43 blvd du 11 novembre 1918, 69622 Villeurbanne Cedex, France.}
\email{celine.vial@univ-lyon1.fr}

\begin{document}

\begin{abstract}

We consider a real Gaussian process $X$ having a global unknown smoothness $(\ro,\bo$), $\ro\in \n_0$ and $\bo \in]0,1[$, with $X^{(\ro)}$ (the mean-square derivative of $X$ if $\ro\ge 1$) supposed to be locally stationary with index $\bo$. From the behavior of quadratic variations built on divided differences of $X$, we derive an estimator of $(\ro,\bo)$ based on  -~not necessarily equally spaced~- observations of $X$. Various numerical studies of these estimators exhibit their properties for finite sample size and different types of processes, and  are also completed by two examples of application to real data.

\end{abstract}


%\MSC{62M05, 62M09, 62F12, 60G15}
\maketitle

\section{Introduction}
In many areas assessing the regularity of a Gaussian process represents still and always an important issue. In a straightforward way, it allows to give accurate estimates for approximation or integration of sampled process. An important example is the kriging, which becomes more and more popular with growing number of applications. This method consists in interpolating a Gaussian random field observed only in few points. Estimating the covariance function is often the first step before plug this estimates in the Kriging equations, see \citet{St00}. Usually the covariance function is assumed to belong to a parametric family, where these unknown parameters are linked to the sampled path regularity: for example the power model, which corresponds to a Fractional Brownian motion. Actually, many applications make use of irregular sampling and \citet{St00} (chap. 6.9) gives an hint of how adding three points very near from the origin among the already twenty equally spaced observations improve drastically the estimation of the regularity parameter. In this paper, we defined an estimator of global regularity of a Gaussian process when the sampling design is regular, that is observation points correspond to quantile of some distribution, see section~\ref{Framework} for details. Taking into account a non uniform design is innovating regarding other existing estimators and makes sense as to the remark above.

A wide range of methods have been proposed to reconstruct a sample path from discrete observations. For
processes satisfying to the so-called  Sacks and Ylvisaker (SY) conditions, recent works include:
\citet[orthogonal projection, optimal designs]{MG96}, \citet[linear interpolation, optimal designs]{MGR97},
\citet[linear interpolation, adaptive designs]{MGR98}. Under H\"{o}lder type conditions, one may cite e.g. works of
\citet[linear interpolation]{Se96}, \citet[Hermite interpolation splines, optimal designs]{Se00}, \citet[best approximation order]{SB98}. Note that a more detailed survey may be found in the book by \citet{Rit00}. Another important topic, involving the knowledge of regularity and arising in above cited works, is the search of an optimal design. In time series context, \citet{Ca85}  analyzes three important problems (estimation of regression coefficients, estimation of random integrals and detection of signal in noise) for which he is looking for optimal design. The latter two problems involve approximations of integrals, where knowledge of process regularity is particularly important \citep[see, e.g.][]{Rit00}, we provide a detailed discussion on this topic  in section~\ref{AppInt} together with additional references. Applications of estimation of regularity can be also find in \citet{Ad90} where bounds of suprema distributions depend on the sample roughness, in \citet{Is92} where the regularity is involved in the choice of the best wavelet base in image analysis or more generally in prediction area, see \citet{Cu77,Li79,Bu85}.

Furthermore for a real stationary and non differentiable Gaussian process with covariance
$\mathds{K}(s,t)=\mathds{K}(\abs{t-s},0)$ such that  $\dsp \mathds{K}(t,0) = \mathds{K}(0,0) -  A
\abs{t}^{2\bo} + o(\abs{t}^{2\bo})$ as $\abs{t} \to 0$, the parameter  $\bo$, $0 < \bo< 1$,  is closely
related to  fractal dimension  of the sample paths. This relationship is developed in particular in the
works by \citet{Ad81} and \citet{TT91} and it gave rise to an important literature around estimation  of
$\bo$. Note that this relation can be extended, e.g. for non Gaussian process in \citet{HR94}.
The recent paper of \citet{Gneit2012} gives a review on estimator of the fractal dimension for times series and spatial data. They also provide a wide range of application in environmental science, e.g. hydrology, topography of sea floor.
Note that this paper is restricted to the case of $(n+1)$ equally spaced observations. In relation with our work, we refer especially to \citet{CH94}  for estimators based on
quadratic variations and their extensions developed by  \citet{KW97}. Still for this stationary framework, \citet{CHP95}
introduce a periodogram-type estimator whereas \citet{FHW94}  use the number of level crossings.

In this paper, our aim is to estimate the global smoothness  $(\ro,\bo)$ of a Gaussian process $X$,
supposed to be $\ro$-times differentiable (for some nonnegative integer $\ro$) where $X^{(\ro)}$ (the $\ro$-th mean-square derivative of $X$ for non-zero $\ro$) is supposed to be locally stationary with regularity $\bo$. The parameters $(\ro,\bo)$ being both unknown,  we improve the previous works in several ways:
\begin{itemize}
\item[-] not necessarily equally spaced  observations of $X$ over some finite interval $[0,T]$ are considered,
\item[-] X is not supposed to be stationary not even with stationary increments,
\item[-] X has an unknown degree of differentiability, $\ro$,  to be estimated,
\item[-] for $\ro \ge 1$, the coefficient of smoothness $\bo$  is related to the unobserved derivative $X^{(\ro)}$.
\end{itemize}
Our methodology is based  on an estimator of $\ro$, say $\widehat{\ro}$, derived from quadratic variations of divided differences of $X$ and consequently, generalize the estimator studied  by \citet{BV11} for the equidistant case. In a second step, we proceed to the estimation of $\bo$, with an estimator $\widehat{\beta}_0$ which can be viewed as a simplification of that studied,  in the case $\ro=0$,  by \citet{KW97}.  Also for processes with stationary increments and using a linear regression approach, \citet{IL97} have proposed and studied an estimator of $H=2(\ro+\bo)$ for equally spaced observations.  As far as we can judge, our two steps procedure seems to be simpler and more competitive. We obtain an upper bound for $\p(\widehat{\ro} \not = \ro)$ as well as the mean square error of $\widehat{\ro}$  and almost sure rates of convergence of $\widehat{\bo}$. Surprisingly, these almost sure rates  are comparable to those obtained in the case of $\ro$ equal to 0: by this way the preliminary estimation of $\ro$ does not affect that of $\bo$, even if $X^{(\ro)}$ is not observed. Next, in section~\ref{AppInt}, we derive theoretical and numerical results concerning the important application of approximation and integration.  We complete this work with an extensive computational study: we compare different estimators of $\ro$ and $\ro+\bo$ for processes with various kinds of smoothness, derive properties of our estimators for finite sample size,  an example of consequence of the misspecification of $\ro$ is given, and an example of process with trend is also study. To end this part, we apply our global estimation to two well-known real data sets: Roller data~\citet{La94} and Biscuit data~\citet{BFV01}.

\section{The framework}\label{Framework}

\subsection{The process and design}
We consider a Gaussian process $ X=\{X(t), \, t\in [0,T] \}$ observed at $(n+1)$ instants on $[0,T]$, $T>0$, with covariance function $\mathds{K}(s,t)=\cov(X(s),X(t))$. We shall assume the following conditions on regularity of $X$.

\begin{ass}[A\ref{h21}]\label{h21}
$X$ satisfies the following conditions.

\begin{itemize}
\item[(i)]There exists some nonnegative integer $\ro$, such that $X$ is $\ro$-times differentiable in quadratic mean, denote $X^{(\ro)}$.
\item[(ii)]The process $X^{(\ro)}$ is supposed to be locally stationary:
\begin{equation} \label{e21}
\lim_{h\to 0} \sup_{s,t\in[0,T],\\\abs{s-t} \le h,\\s\not=t} \abs{\frac{\esp\big(X^{(\ro)}(s)
-X^{(\ro)}(t)\big)^2}{\abs{s-t}^{2\bo}} - d_0(t)} = 0
\end{equation}
where $\bo\in ]0,1[$ and $d_0$ is a positive continuous function on $[0,T]$.
\item[(iii-{$ p$})]For either $p=1$ or $p=2$, $\mathds{K}^{(\ro+p,\ro+p)}(s,t)$
exists  on $[0,T]^2\big\backslash\{s=t\}$ and satisfies  for some $D_p>0$:
\begin{equation*}
\abs{\mathds{K}^{(\ro+p,\ro+p)}(s,t)} \le D_p\abs{s-t}^{-(2p- 2\bo)}.
\end{equation*}
Moreover, we suppose that  $\mu\in C^{\ro+1}([0,T])$.
\end{itemize}
\end{ass}
Note that the local stationarity makes reference to \citet{Ber74}'s sense. The condition A\ref{h21}-(i) can be translated in terms of the covariance function. In particular the function $\mathds{K}$ is continuously differentiable with derivatives $\mathds{K}^{(r,r)}(s,t)= \cov (X^{(r)}(s), X^{(r)}(t)) $, for
$r=1,\dotsc,\ro$. Also, the mean of the process $\mu(t):=\esp X(t)$ is a
$\ro$-times continuously differentiable function with $\esp X^{(r)}(t)= \mu^{(r)}(t)$, $r=0,\dotsc,\ro$. Conditions~A\ref{h21}-(iii-$p$) are more technical but classical ones when estimating regularity parameter, see \citet{CH94,KW97}.

These assumptions are satisfied by a wide range of examples, e.g. the $\ro$-fold integrated fractional Brownian motion or the Gaussian process with Mat\'ern covariance function, i.e. $\mathds{K}(t,0)=\frac{\pi^{1/2}\phi}{2^{(\nu)-1}\Gamma(\nu+1/2)}(\alpha|t|)^{\nu}K_{\nu}(\alpha|t|)$, where $K_{\nu}$, is a modified Bessel function of the second kind of order $\nu$. The latter process gets a global smoothness equal to $(\lfloor\nu\rfloor,\nu-\lfloor\nu\rfloor)$, see \citet{St00} p.31. Detailed examples, including different classes of stationary processes, can be found in~\citet{BV08,BV11}.

Note that, for processes with stationary increments, the function $d_0$ is reduced to a constant. Of course, cases with non constant $d_0(\cdot)$ are allowed as well as processes  with a smooth enough trend. In particular, for some sufficiently smooth functions $a$ and $m$ on $[0,T]$, the process $Y(t)=a(t)X(t)+m(t)$ will also fulfills Assumption~A\ref{h21}, see lemma~\ref{l61} for details.

Let us turn now to the description of observation points. We consider that the process $X$ is observed at $(n+1)$ instants, denoted by $$0=t_{0,n} < t_{1,n} < \dotsb < t_{n,n}\le T$$
where the $t_k := t_{k,n}$ form a regular sequence design. That is, they  are defined as quantiles of a fixed positive and continuous density $\psi$ on $[0,T]$:
\begin{equation*}
\int_{0}^{t_k} \psi(s) \mathrm{d}s =  \frac{k\dn}{T}, \;\;\; k=0,\dotsc,n,
\end{equation*}
for $\dn$  a positive sequence such that $\dn\to 0$ and $n\dn\to T(-)$. Clearly,  if $\psi$ is the uniform density on $[0,T]$, one gets the equidistant case. Some further assumptions on $\psi$ are needed to get some control over the $t_k$'s.
\begin{ass}[A\ref{h22}]\label{h22} The density $\psi$ satisfies:
\begin{itemize}
\item[(i)] $\dsp \inf_{t\in[0,T]} \psi(t)>0$,
\item[(ii)]  $\fa (s,t)\in [0,T]^2$, $\abs{\psi(s) - \psi(t)} \le L \abs{s-t}^{\alpha}$, for some $\alpha \in ]0,1]$.
\end{itemize}
\end{ass}

These hypothesis ensure a controlled spacing between two distinct points of observation, see Lemma~\ref{l62}. From a practical point of view, this flexibility may allow to recognize inhomogeneities in the process (e.g. presence of pics in environmental
pollution monitoring, see \citet{Gi87} and references therein) or else to describe situations where data are collected at equidistant
times but become irregularly spaced after some screening (see for example the wolfcamp-aquifer data in \citet{Cr93} p. 212).

\subsection{The methodology}
In this part, we want to give background ideas about the construction of our estimate of the global regularity $(\ro,\bo)$, when the process is observed on a non equidistant grid.
The main idea is to introduce divided differences, quantities generalizing the finite differences, studied by \citet{BV11,BV12}. Let us first recall that the unique polynomial of degree $r$ that interpolates a function $g$ at $r+1$ points $t_k,\dotsc,t_{k+ur}$ (for some positive integer $u$) can be written as:
\begin{multline}
g[t_k] + g[t_k,t_{k+u}](t-t_k) + g[t_k,t_{k+u},t_{k+2u}](t-t_k)(t-t_{k+u})  \\+  \dotsb + g[t_k,\dotsc,t_{k+ru}](t-t_k)\dotsm (t-t_{k+(r-1)u})
\label{e22}\end{multline}
where the divided differences $g[\dots]$ are defined by
$g[t_k] = g(t_k)$ and for $j=1,\dotsc,r$ (using the Lagrange's representation):
$$
g[t_k,\dotsc,t_{k+ju}] = \sum_{i=0}^j \frac{g(t_{k+iu})}{\prod_{m=0,m\not=i}^j (t_{k+iu} - t_{k+mu})}.
$$
In particular, we  write $g[t_k,\dotsc,t_{k+ru}] = \sum_{i=0}^r b_{ikr}^{(u)} g(t_{k+iu})$ with
\begin{equation} \label{e23}
b_{ikr}^{(u)} := \frac{1}{\prod_{m=0,m\not=i}^r (t_{k+iu} - t_{k+mu})}.
\end{equation}

These coefficients are of particular interest. In fact their first non-zero moments are of order $r$. We can also derive an explicit bound and  an asymptotic expansion for $b_{irk}^{(u)}$, see lemma~\ref{l63} for details.

Then, for positive integers $r$ and $u$, we  consider the $u$-dilated divided differences of order $r$ for $X$:
\begin{equation} \label{e24}
D_{r,k}^{(u)}\,X=\sum_{i=0}^{r} b_{ikr}^{(u)}X(t_{k+iu}), \;\;\; k=0,\dotsc,n- u r
\end{equation}
with $b_{ikr}^{(u)}$ defined by \eqref{e23}. Note that, if  $\psi(t) = T^{-1}\indi_{[0,T]}(t)$, the sequence of designs is equidistant, that is $t_{k,n}=k\delta_n$, and divided differences turn to be finite differences. More precisely, for the sequence $a_{i,r}= \binom{r}{i} (-1)^{r-i}$, let define the finite differences $\Delta_{r,k}^{(u)} = \sum_{i=0}^r a_{i,r} X((k+iu)\dn)$. Noticing that in the case of equally spaced observations,
$$b_{ikr}^{(u)} = \frac{(u\dn)^{-r}}{\prod_{m=0,m\not= i}^{r} (i-m)} = \frac{a_{i,r}}{r!}(u\dn)^{-r},$$
 we deduce the relation $D_{r,k}^{(u)}\,X= \frac{(u\dn)^{-r}}{r!} \Delta_{r,k}^{(u)}\,X$.
From now on, we set the $u$-dilated quadratic variations of $X$ of order $r$: $\dsp \overline{\big( D_r^{(u)}X\big)^2}=  \frac{\sum_{k=0}^{n_r} \big( D_{r,k}^{(u)}\,X\big)^2}{n_r+1}$ with $n_r := n-ur$. Construction of our estimators is based on following asymptotic properties concerning the mean behavior  of $\overline{\big( D_r^{(u)}X\big)^2}$:
%\;\;\text{ and }\;\; \esp\overline{\big( D_r^{(u)}X\big)^2}=  \frac{\sum_{k=0}^{n_r} \esp\big(D_{r,k}^{(u)}\,X\big)^2}{n_r+1}.$$

\medskip
$$(P)\left\{\begin{tabular}{l}\textrm{the quantity $\esp\, \overline{\big(D_{r} ^{(u)}X\big)^2}$ is of order $(u/n)^{-2(p-\bo)}$ for $r=\ro+p$} \\
\textrm{for $p=1,2$, and gets a finite non zero limit when $r\leq \ro$.}
\end{tabular}\right.
$$

See proposition~\ref{p61} for precise results.
These results imply that a good  choice of $r$ (namely $r=\ro+1$ or $\ro+2$) could provide an estimate of $\bo$, at least with an adequate combination of $u$-dilated quadratic variations of $X$. To this end, we propose a two steps
procedure:
\begin{itemize}
\item[Step 1:]\textit{Estimation of $\ro$.}\\ Based on $D_{r,k}^{(1)}X$, we estimate $\ro$ with
\begin{equation} \label{e25}
\widehat{\ro} =\min\Big\{r\in\{2,\dotsc,m_n\} : \overline{\big( D_r^{(1)}X\big)^2}\geq
 n^2b_n\Big\}-2.
\end{equation}
If the above set is empty, we fix
%$\{\dots\}= \emptyset$, we set
$\widehat{\ro}= l_0$ for an arbitrary value $l_0\not\in {\n}_{\sste 0}$.  Here, $m_n\to\infty$ but if an upper bound $B$ is known for
$\ro$, one has to choose $m_n=B+2$. The threshold $b_n$ is a positive sequence chosen such that : $n^{-2(1-\bo)}b_n \rightarrow 0$ and $n^{2\bo}b_n\to \infty$
for all $\bo \in ]0,1[$. For example, omnibus choices are given by $b_n=(\ln n)^{\alpha}$,
$\alpha\in\reel$.

\item[Step 2:]\textit{Estimation of $\bo$.}\\
 Next, we derive two families of estimators for $\bo$, namely $\widehat{\beta}_n^{(p)}$, with either $p=1$ or
    $p=2$ and $u,v$ $(u<v)$  given integers:
$$
\widehat{\beta}_n^{(p)}:= \widehat{\beta}_n^{(p)}(u,v)= p+ \frac{1}{2}\frac{ \ln\Big(\overline{\big(
D_{\widehat{\ro}+p}^{(u)}X\big)^2}\Big) -\ln\Big( \overline{ \big( D_{\widehat{\ro}+p} ^{(v)}X\big)^2}\Big)}{\ln(u/v)}.
$$
\end{itemize}
{\rmrk[The case of $\ro=0$ with equally spaced observations] \label{RqKW}  \citet{KW97} proposed estimators of $2\bo=\alpha$ based on ordinary and generalized least squares on the logarithm of the quadratic variations versus logarithm of a vector of values $u$, more precisely
$$
\hat{\alpha}^{(p)}=\frac{(\vone\T W \vone)(\vu\T W \vQ^{(p)})-(\vone\T W \vu)(\vone\T W \vQ^{(p)})}{(\vone\T W \vone)(\vu\T W \vu)-(\vone\T W \vu)^2}
$$
where $\vone$ is the $m$-vector of $1$s, $\vu=(\ln(u),u=1,\ldots, m)\T$, $\vQ^{(p)}=(\ln(\overline{\big(\Delta_{p}^{(u)}X\big)^2})$, $u=1,\ldots, m)\T$ and $W$ is either the identity matrix $I_m$ of order $m\times m$ or a matrix depending on $(n,\bo)$ which converges to the asymptotic covariance of $n^{1/2}\big(\overline{(\Delta_{p}^{(u)}X)^2}-\esp\overline{(\Delta_{p}^{(u)}X)^2}\big)$. The ordinary least square estimator -- corresponding to $W=I_m$-- is denoted by $\widehat{\alpha}_{OLS}^{(p)}$, where $p$ is adapted to the regularity of the process (supposed to be known in their work). The choice $p=1$, with the sequence $(-1,1)$, leads to the estimator studied by \citet{CH94}. Remark that, for $(u,v)=(1,2)$, one gets $\widehat{\beta}_n^{(1)}= \widehat{\alpha}^{(0)}_\text{OLS}$ and $\widehat{\beta}_n^{(2)}= \widehat{\alpha}^{(1)}_\text{OLS}$ but, even in this equidistant case, new estimators may be derived with other choices of $(u,v)$ such as $(u,v)=(1,4)$ (which seems to perform well, see Section~\ref{sub52}).}

\section{Asymptotic results}\label{Asymptres}
In \citet{BV11}, an exponential bound is obtained for $\p( \widehat{\ro} \not= \ro)$ in the equidistant case, implying that, almost surely for $n$ large enough, $\widehat{\ro}$  is equal to $\ro$. Here, we generalize this result to regular sequence designs but also,  we complete it  with the average behavior of $\widehat{\ro}$.
\begin{theo} \label{t31}
Under Assumption A\ref{h21} (fulfilled with $p=1$ or $p=2$) and A\ref{h22}, we have
$\dsp \p(\widehat{\ro} \not= \ro) = {\cal O} \Big( \exp\big( - \vp_n(p)\big)\Big)$  and $\dsp\esp(\widehat{\ro} - \ro)^2 = {\mathcal O} \Big( m_n^3 \exp\big(- \vp_n(p) \big)\Big)$, where, for some positive constant $C_1(\ro)$, $\vp_n(p)$ is defined by
$$\vp_n(p) =  C_1(\ro)\times \begin{cases}    \dsp n\indi_{]0,\frac{1}{2}[}(\bo)+ n(\ln n)^{-1}
\indi_{\{\frac{1}{2}\}}(\bo)+n^{2-2\bo}\indi_{]\frac{1}{2},1[}(\bo) \text{ if } p=1\\ \dsp n \text{ if } p=2.\end{cases}$$
 \end{theo}
Remark that one may choose $m_n$ tending to infinity arbitrary slowly. Indeed, the unique restriction is that $\ro$ belongs to the grid $ \{1,\dotsc,m_n\}$ for $n$ large enough. From a practical point of view, one may choose a preliminary fixed bound $B$, and, in the case where the estimator return the non-integer value $l_0$, replace $B$ by $B'$ greater than  $B$.

\bigskip

The bias of $\h$ will be controlled by a second-order condition of local stationarity, more specifically  we
have to strengthen the relation \eqref{e21} in:
\begin{equation} \label{e31}
\lim_{h\to 0} \sup_{\substack{s,t\in[0,T],\\\abs{s-t} \le h,\\s\not=t}} \abs{\,
\abs{s-t}^{-\bun}\Big(\frac{\esp\big(X^{(\ro)}(s)
-X^{(\ro)}(t)\big)^2}{\abs{s-t}^{2\bo}} - d_0(t)\Big)-d_1(t)} = 0
\end{equation}
for a positive $\bun$ and continuous function $d_1$.

\begin{theo}\label{t32} If relation \eqref{e31}, Assumption A\ref{h21} with $p=1$ or $p=2$, and A\ref{h22} are fulfilled, we obtain
$$\limsup_{n\to\infty} V_n^{(p)} \abs{\h - \bo} \le C_1(p)\;\;\;\text{a.s.}$$
where $C_1(p)$ is some positive constant and
\begin{align*}
V_n^{ (1) }&= \min\Big(n^{\bun},\sqrt{\frac{n}{\ln n}} \indi_{]0,\frac{3}{4}[}(\bo)
+\frac{\sqrt{n}}{\ln n}\indi_{\{\frac{3}{4}\}}(\bo)+ \frac{n^{2(1-\bo)}}{\ln
n}\indi_{]\frac{3}{4},1[}(\bo)\Big),\\
V_n^{ (2) }&= \min\Big(n^{\bun},\sqrt{\frac{n}{\ln n}} \Big).
\end{align*}
\end{theo}

{\rmrk[Rates of convergence with equally spaced observations] For stationary gaussian processes, \citet{KW97} give the mean square error and convergence in distribution of their estimator described in remark~\ref{RqKW}.  They obtained the same rate up to a logarithmic order, due here to almost sure convergence, for both families $p=1$ and $p=2$. The asymptotic distribution is either of Gaussian or of Rosenblatt type depending on $\bo$ less or greater than $3/4$. \citet{IL97} introduced an estimator of $H=2(\ro+\bo)$ (for stationary increment processes)  with a global linear regression approach, based on an asymptotic equivalent of the quadratic variation and using some adequate family of sequences. For $\ro=0$, their approach matches with the previous one, with $p=1$, in using dilated sequence of type $a_{jr}= \binom{r}{j} (-1)^{r-j}$. Assuming a known upper bound on $\ro$, they derived convergence in distribution to a centered Gaussian variable with rate depending on $\bo$--root-$n$ for $\bo\leq3/4$ and $n^{1/2-\alpha(2\bo-3/2)}$, where $\delta_n=n^{-\alpha}$ and $\alpha<1$, for $\bo>3/4$. In this last case, to obtain a gaussian limit they have to assume that $\ro$ is known, the observation interval is no more bounded, and the rate of convergence is lower than root-$n$.}

\section{Approximation and integration }\label{AppInt}
\subsection{Results for plug-in estimation}

A classical and interesting topic is  approximation and/or integration  of a sampled path. An extensive literature may be found on these topics with a detailed overview in the recent monograph of \citet{Rit00}. The general framework is as follows : let
$ X=\{X_t, \, t\in [0,1] \}$,  be observed at sampled times $t_{0,n},\dotsc,t_{n,n}$ over
[a,b], more simply denoted by $t_{0},\dotsc,t_{n}$. Approximation  of $X(\cdot)$ consists in interpolation of the  path on $[a,b]$,  while weighted integration is the calculus of $\mathcal{I}_{\rho}=\int_a^b X(t) \rho(t) \dt$ for some positive and continuous weight function $\rho$. These problems are closely linked, see e.g. \citet{Rit00} p. 19-21. Closely to our framework of local stationary derivatives, we may refer more specifically to works of \citet{PRW04} for approximation and \citet{Be98} for integration. For sake of clarity, we give a brief summary of their obtained results. In the following, we denote by ${\cal H}(\ro,\bo)$ the family of Gaussian processes having $\ro$ derivatives in quadratic mean and $\ro$-th derivative with H\"{o}lderian regularity of order $\bo\in]0,1[$.  For
measurable  $g_i(\cdot)$,  we consider the approximation  ${\cal A}_{n, g} (t) = \sum_{i=0}^n
X(t_i) g_i(t)$ and the corresponding weighted and integrated $L^2$-error $e_{\rho}({\cal A}_{n,g})$  with $e_{\rho}^2({\cal A}_{n,g}) = \int_a^b \esp\abs{ X(t) - {\cal A}_{n, g}(t)}^2 \rho(t)\dt$. For $X\in{\cal H}(\ro,\bo)$ and known $(\ro,\bo)$, \citet{PRW04} have shown that
\begin{multline*}0< c(\ro,\bo)  \le \varliminf_{n\to\infty} n^{\ro+\bo}\inf_{g} e_{\rho}({\cal A}_{n,g})
\\ \le \varlimsup_{n\to\infty} n^{\ro+\bo} \inf_{g}  e_{\rho}({\cal A}_{n,g}) \le C(\ro,\bo)<+\infty
\end{multline*}
for equidistant sampled times $t_1,\dotsc,t_n$ and Gaussian processes defined and observed on the half-line $[0,+\infty[$. Of course, optimal choices of  functions $g_i$, giving a minimal error, depend on the unknown covariance function of $X$.

For weighted integration, the quadrature is denoted by ${\cal Q}_{n, d}= \sum_{i=0}^n
X(t_i) d_i$ with well-chosen constants $d_i$ (typically, one may take $d_i = \int_a^b g_i(t)\dt$). For known $(\ro,\bo)$, a short list of references could be:
\begin{itemize}
\item[-] \citet{SY68,SY70} with $\ro=0$ or 1, $\bo=\frac{1}{2}$ and known covariance,
\item[-] \citet{BC92} for arbitrary $\ro$ and $\bo=\frac{1}{2}$,
\item[-] \citet{St95} for stationary processes and $\ro+\bo<\frac{1}{2}$,
\item[-] \citet{Ri96} for minimal error, under Sacks and Ylvisaker's  conditions, and with arbitrary $\ro$.
\end{itemize}
Let us set $e^{2}_{\rho}({\cal
Q}_{n,d})=\esp\abs{ I_{\rho} - {\cal Q}_{n, d}}^2 $, the  mean square error of integration. In the stationary case and for known $\ro$, \citet{Be98} established the following exact behavior: If $\rho \in C^{\ro+3}([a,b])$ then for some given quadrature ${\cal Q}_{n,d^{\ast}(\ro)}$ on $[a,b]$,
$$n^{\ro+\bo+\frac{1}{2}}\,  e_{\rho}({\cal Q}_{n,d^{\ast}(\ro)})\xrightarrow[n\to\infty]{}c_{\ro,\bo} (\int_a^b \rho^2(t) \psi^{-(2(\ro+\bo)+1)}(t)\dt)^{\frac{1}{2}}$$
where $\psi$ is the density relative to the regular sampling $\{t_{1},\dotsc,t_n\}$. Moreover, following \citet{Ri96}, it appears that this last result is optimal under Sacks and Ylvisaker's conditions. Finally, \citet{IL97b} have proposed a quadrature, requiring only an upper bound on $\ro$, also with an error of order $ {\cal O} \big( n^{-(\ro+\bo+\frac{1}{2})} \big)$.

\medskip\

All these results shown the importance of well estimating $\ro$ and motivate ourself to focus on plugged-in interpolators, namely those using
 Lagrange polynomial of
order estimated by $\widehat{\ro}$. More precisely, Lagrange interpolation of order  $r\ge 1$ is defined by
\begin{equation}\label{e41}
\wt{X}_r(t)=\sum_{i=0}^{r} L_{i,k,r}(t) X \big( t_{kr+i} \big),\text{ with } L_{i,k,r}(t) =
\prod_{\substack{j=0\\j\not=i}}^{r}\frac{(t- t_{kr+j})}{t_{kr+i} - t_{kr+j}},
\end{equation}
for $\dsp t\in {\cal I}_k := \Big[ t_{kr}, t_{kr+r}\Big]$, $\dsp k=0,\dotsc,\lfloor \frac{n}{r}\rfloor-2$ and $\dsp{\cal I}_{\lfloor \frac{n}{r}\rfloor-1} = \Big[ \lfloor t_{\lfloor(\frac{n}{r}\rfloor-1)r},T\Big]$.

%In other words , we set $b_i (\cdot)= L_{i (\text{ mod }r),k,r}(\cdot)$ for $t\in {\cal I}_k$ and $d_i = \int_0^T L_{i (\text{ mod }r),k,r}(t)\dt$.
Our plugged method will consist in the approximation given by ${\cal A}_{n,L}(t) = \wt{X}_{\max(\widehat{\ro},1)}(t)$, $t\in [0,T]$, and quadrature  by ${\cal Q}_{n,L} = \int_0^T\wt{X}_{\widehat{\ro}+1}(t)\,\rho(t)\dt$.  Indeed, Lagrange polynomials are of easy implementation and by the result of \citet{PRW04} with known $\ro$,  they reach the optimal rate of approximation, $n^{-(\ro+\bo)}$ without requiring knowledge of covariance. Our following result shows also that the associate quadrature has the expected rate $n^{-(\ro+\bo+\frac{1}{2})}$.  Indeed, in the weighted case and for $T>0$, we  obtain the following asymptotic bounds in the case of a regular design.

\begin{theo} \label{t41} Suppose that conditions A\ref{h21}(i)-(ii) and A\ref{h22} hold, choose a logarithmic order for $m_n$ in \eqref{e25} and consider a positive and continuous weight function $\rho$.
\begin{itemize}
\item[(a)] Under condition A\ref{h21}(iii-1),  we have
$$
e_{\rho} ({\rm{app}}\big(\widehat{\ro})\big) := \Big(\int_0^T \esp\abs{X(t) - \tX_{\max(\widehat{\ro},1)}(t)}^2 \rho(t) \dt\Big)^{1/2}
={\cal O}\big(n^{-(\ro+\bo)}\big),
$$

\item[(b)] if condition  A\ref{h21}(iii-2) holds:
$$
e_{\rho} ({\rm{int}}\big(\widehat{\ro})\big):= \Big(\esp\abs{\int_0^T  (X(t) - \tX_{\widehat{\ro}+1}(t) ) \rho(t) \dt}^2\Big)^{1/2} ={\cal O} \big(n^{-(\ro+\bo +\frac{1}{2})}\big).
$$
\end{itemize}
\end{theo}
In conclusion, expected rates for approximation and integration are reached by plugged Lagrange piecewise polynomials. Of course if  $\ro$ is known,  this last result holds true with $\widehat{\ro}$ replaced by $\ro$.

\subsection{Simulation results}

The figure~\ref{figinterpol} is obtained using $1000$ simulated sample paths observed in equally spaced points on $[0,1]$. This figure illustrates results of approximation for different processes. The logarithm of empirical integrated mean square error (in short IMSE), i.e. $e_{1}^2(\rm{app}\big(\widehat{\ro})\big)$,  is drawn as a function of $\ln(n)$ with a range of sample size from $ 25$ to $1000$. We may notice that we obtain straight lines with slope very near to $-H=-2(\ro+\bo)$. Since the Ornstein-Uhlenbeck process is a scaled time-transformed Wiener process, intercepts are different contrary to stationary versus non-stationary continuous ARMA processes.

\medskip

\begin{figure}[!htm]
\begin{center}
{\includegraphics[width=9cm]{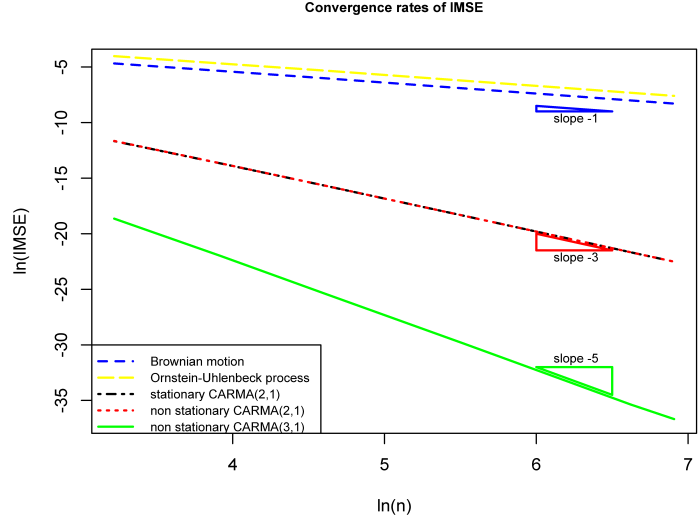}}
\end{center}
\caption{Logarithm of $e_{1}^2(\rm{app}\big(\widehat{\ro})\big)$, i.e. the IMSE in function of $\ln(n)$, for different  processes.
The dashed line corresponds to Brownian motion, long dashed line to O.U., dashed line to non stationary CARMA(2,1), dotted-dashed line to stationary CARMA(2,1) and solid line to non stationary CARMA(3,1).
The small triangles near lines are here to indicate the theoretic slope.} \label{figinterpol}
\end{figure}

%%%%%%%%%%%%%%%%%%%%%%%%%%%%%%%%%%%%%%%%%%%%%%%%%%%%%%%%%%%%%%%%%%%
%
%%%%%%%%%%%%%%%%%%%%%%%%%%%%%%%%%%%%%%%%%%%%%%%%%%%%%%%%%%%%%%%%%%%

\section{Numerical results}
In this section, to numerically compare our estimators with existing ones, we restrict ourselves to the equidistant case with the choice $\psi(t) = \frac{1}{T} \indi_{[0,T]}(t)$. As noticed before, we get  for $a_{i,r}= \binom{r}{i} (-1)^{r-i}$ and $\Delta_{r,k}^{(u)} = \sum_{i=0}^r a_{i,r} X((k+iu)\dn)$,  the relation $D_{r,k}^{(u)}\,X= \frac{(u\dn)^{-r}}{r!} \Delta_{r,k}^{(u)}\,X$ implying in turn that
\begin{equation} \label{e51}
\widehat{H}_n^{(p)} = \frac{ \ln\Big(\overline{\big(
\Delta_{\widehat{\ro}+p}^{(u)}X\big)^2}\Big) -\ln\Big( \overline{ \big( \Delta_{\widehat{\ro}+p}
^{(v)}X\big)^2}\Big)}{\ln(u/v)}
\end{equation}
is a consistent estimator of $H=2(\ro +\bo)$. All the simulation results are obtained by simulation of trajectories using two different methods : for stationary processes or with stationary increments we use the procedure described in~\citet{WC94} and for CARMA (continuous ARMA) processes, we use~\citet{TC00}. Each of them consists in $n$ equally spaced observation points on $[0,1]$ and 1000 simulated sample paths. All computations have been performed with the R software \citep{R12}.

\subsection{Results for estimators of $\ro$}

This section is dedicated to the numerical properties of two estimators of $\ro$.
We consider the estimator introduced by \citet{BV11}, derived from~\eqref{e25} in the equidistant case. An alternative, says $\tilde r_n$, based on Lagrange interpolator polynomials  was proposed by \citet{BV08}. More precisely, for $\dn=n^{-1}$ et $T=1$, $\tilde r_n$ is  defined by
\begin{multline*}
\tilde{r}_n =\min\Big\{r\in\{1,\ldots,m_n\} :
\frac{1}{r\wt{n}_r}\sum_{k=0}^{r\wt{n}_r -1} \big(X\Big( {
\frac{2k+1}{n}}\Big)-\tXr\Big({\frac{2k+1}{n}}\Big)\big)^2 \\ \geq n^{-2r} b_n\Big\}-1
\end{multline*}
where $\wt{n}_r=  \lfloor \frac{n}{2r} \rfloor$ and $\tXr(s)$ is defined for all $s \in [0,1]$ and each $r\in \{1,\dotsc,m_n\}$ in the following way : there
exist $k=0,\dotsc, \wt{n}_r-1$ such that for  $t\in
{\mathcal I}_{2k} := [\frac{2kr}{n}, \frac{2(k+1)r}{n}]$, the piecewise Lagrange interpolation of $X(t)$,  $\wt{X}_r(t)$, is given by
$\wt{X}_r(t)=\sum\limits_{i=0}^{r} L_{i,k,r}(t) X \big( (kr+i)n^{-1} \big)$, with  $L_{i,k,r}(t)=
\prod_{\substack{j=0\\j\not=i}}^{r}\frac{(t- (kr+j)n^{-1})}{(i-j)n^{-1}}$.

Both estimators use the critical value $b_n$ which is involved in detection of the jump. Here, due to convergence properties, we make the choice $b_n = (\ln n)^{-1}$. Table~\ref{tabNS} illustrates the strong convergence of both estimators and shows that this convergence is
valid even for small number of observation points $n$, up to 10 for the estimator $\widehat{\ro}$. We may noticed that, in the case of bad estimation, our estimators overestimate the number of derivatives. Remark also that, for identical sample paths, $\widehat{\ro}$ seems to be more robust than $\tilde r_n$. This behavior was expected as the latter uses only half of the observations for the detection of the jump in quadratic mean.  In these first results,  processes have fractal index $\bo$ equals to $1/2$, but alternative choices of $\bo$ are  of interest, so we consider  the fractional Brownian motion (in short fBm) and the integrated fractional Brownian motion (in short ifBm), with respectively $\ro=0$ and $\ro=1$ and various values of $\bo$.

\begin{table}[t]
%\medskip
\caption{{\small Value of the empirical probability that $\widehat{\ro}$ or $\tilde r_n$ equals $\ro$ or $\ro+1$
%using $1000$ simulated sample paths for tree different type of non stationary processes. Each realization consists in
 with $n=10$ or $25$.
 % equally spaced observations on $[0,1]$.
\label{tabNS}}}

\medskip

\centering
%{\normalfont \footnotesize
\begin{tabular}{|c|cc|cc|cc|}
\cline{2-7}
\multicolumn{1}{c|}{}&\multicolumn{2}{|c|}{Wiener process, $\ro=0$}&\multicolumn{2}{|c|}{CARMA(2,1), $\ro=1$}&\multicolumn{2}{|c|}{CARMA(3,1), $\ro=2$}\\
\hline
&\multicolumn{6}{|c|}{Number of equally spaced observations $n$}\\
event &10&25&10&25&10&25 \\  \hline
$\tilde r_n=\ro$&0.995&1.000&0.913&1.000&0.585&0.999\\
$\tilde r_n=\ro+1$&0.005&0.000&0.087&0.000&0.415&0.001\\\hline
$\widehat{\ro}=\ro$&1.000&1.000&1.000&1.000&0.999&1.000\\
$\widehat{\ro}=\ro+1$&0.000&0.000&0.000&0.000&0.001&0.000\\
\hline
\end{tabular}
%}
\end{table}

Table \ref{tabfBmifBm} shows that $\widehat{\ro}$ succeeds in estimating the true regularity for $\bo$ up to 0.9. Of course the number of observations must be large enough and, even more important for large values of $\ro$ when $\bo\ge 0.95$. This latter result is clearly apparent when one compares the errors obtained for  an ifBm with $\bo=0.95$ and a fBm with $\bo=0.95$. Finally, we can see once more that $\tilde r_n$ is less robust against increasing $\bo$, whereas our simulations have shown that, for $n=2000$ and each simulated path, the estimator $\widehat{\ro}$ is able to distinguish processes with regularity $(0,0.98)$ and $(1,0.02)$, an almost imperceptible difference!

\begin{table}
\caption{{\small Value of the empirical probability that  $\widehat{\ro}$ or $\tilde r_n$ equals $\ro$  for a fractional Brownian motion or an integrated one with fractal index $2\bo$.  \label{tabfBmifBm}}}

\medskip

\centering
\scalebox{0.92}{\begin{tabular}{|c|ccccc|ccccc|}
%&&\multicolumn{10}{|c|}{estimator}\\
\cline{2-11}
\multicolumn{1}{c|}{}&\multicolumn{5}{|c| }{$\widehat{\ro}=\ro$}&\multicolumn{5}{|c|}{$\tilde r_n=\ro$}\\\cline{2-11}
\multicolumn{1}{c|}{}&\multicolumn{10}{|c|}{number of equally spaced observations $n$}\\
\multicolumn{1}{c|}{}&50&100&500&1000&1200&50&100&500&1000&1200\\
\hline
fBm $\bo$&&&&&&&&&&\\
0.90&1.000&1.000&1.000&1.000&1.000&0.655&0.970&1.000&1.000&1.000\\
0.95&0.969&0.999&1.000&1.000&1.000&0.002&0.002&0.004&0.134&0.331\\
0.97&0.242&0.521&1.000&1.000&1.000&0.000&0.000&0.000&0.000&0.000\\
0.98&0.019&0.015&0.0420&0.5258&0.759&0.000&0.000&0.000&0.000&0.000\\
\hline
\hline
ifBm $\bo$&&&&&&&&&&\\
0.02&1.000&1.000&1.000&1.000&1.000&1.000&1.000&1.000&1.000&1.000
%0.10&%1.000&1.000&1.000&1.000&&1.000&1.000&1.000&1.000&
\\\hline
0.90&1.000&1.000&1.000&1.000&1.000&0.000&0.000&0.645&0.999&1.000\\
0.95&0.305&0.888&1.000&1.000&1.000&0.000&0.000&0.000&0.000&0.000\\
0.97&0.000&0.000&0.292&0.993&1.000&0.000&0.000&0.000&0.000&0.000\\
\hline
\end{tabular}}
\end{table}

\subsection{Estimation of $H$ and $\bo$}\label{sub52}
This part is dedicated to the numerical properties of estimators $\widehat H_n^{(p)}$, for $p=1$ or $2$ using the values $u=1$ and $v=4$ (giving more homogeneous results than $u=1$ and $v=2$). It ends with real data examples.

\subsubsection{Quality of estimation}
For the numerical part, we focus on the study of fBm, ifBm and, CARMA(3,1) with $\ro=2$, $\bo=0.5$.  Table~\ref{tabHp} illustrates the performance of our estimators when $\bo$, $\ro$ are increasing: we compute the empirical mean-square error from our 1000 simulated sample paths and $n=1000$ equally spaced observations are considered. It appears that, contrary to $\widehat H_n^{(2)}$,  the estimator $\widehat H_n^{(1)}$ slightly deteriorates for values of $\bo$ greater than 0.8. This result is in agreement with the rate of convergence of Theorem~\ref{t32}, that depends on $\bo$ for this estimator. The bias is negative and seems to be unsensitive to the value of $\ro$ but the mean-square error is slightly deteriorated from $\ro=0$ to $\ro=1$ in both cases. Finally, for $\bo<0.8$, $H_n^{(1)}$ seems preferable to $\widehat H_n^{(2)}$, possibly due to a lower variance of this estimator. Nevertheless, both estimators perform globally well on these numerical experiments.

\begin{table}[t]
\caption{{\footnotesize Values of mean square error and bias (between brackets) for estimators $\widehat H_n^{(p)}$, for $p=1$ or $2$ and $n=1000$.
\label{tabHp}}}

\medskip

\centering{
{\normalfont \footnotesize
\begin{tabular}{|c|ccccc|}
\cline{2-6}
\multicolumn{1}{c|}{}&\multicolumn{5}{|c|}{fBm $\bo$}\\
\multicolumn{1}{c|}{}&0.2&0.5&0.8&0.9&0.95\\ \hline
$\widehat H_n^{(1)}$&\tta{0.0017}{-0.0020}&\tta{0.0019}{-0.0045}&\tta{0.0034}{-0.0120}&\tta{0.0054}{-0.0295}&\tta{0.0065}{-0.0448}\\
$\widehat H_n^{(2)}$&\tta{0.0030}{-0.0026}&\tta{0.0039}{-0.0038}&\tta{0.0040}{-0.0057}&\tta{0.0039}{-0.0069}&\tta{0.0039}{-0.0084}\\
\hline\hline
\multicolumn{1}{c|}{}&\multicolumn{5}{|c|}{ifBm $\bo$}\\
\multicolumn{1}{c|}{}&0.2&0.5&0.8&0.9&0.95\\\hline
$\widehat H_n^{(1)}$&\tta{0.0032}{-0.0072}&\tta{0.0026}{-0.0047}&\tta{0.0041}{-0.0150}&\tta{0.0061}{-0.0342}&\tta{0.0073}{-0.0488}\\
$\widehat H_n^{(2)}$&\tta{0.0055}{-0.0106}&\tta{0.0051}{-0.0060}&\tta{0.0046}{-0.0060}&\tta{0.0044}{-0.0061}&\tta{0.0043}{-0.0061}\\
\hline\hline
\end{tabular}}
}
\end{table}

\subsubsection{Asymptotic properties}

Results of Theorem~\ref{t32} are also illustrated in Table~\ref{tabvitesse} where we have computed the regression of $\ln(\esp|\widehat{H}_n^{(p)} - H|)$ on $\ln n$  for various values of $n$ and $\esp|\widehat{H}_n^{(p)} - H|$ estimated from our 1000 simulated sample paths. As expected, the slope (corresponding to our arithmetical rate of convergence) is constant and approximatively  equal to 0.5 for $\widehat{H}_n^{(2)}$ while, for $\widehat{H}_n^{(1)}$, the decrease is apparent for high values of $\bo$. Finally, Figure~\ref{figboxplot} illustrates the behavior of the estimators $\widehat H_n^{(p)}$ with $p=1$ or $2$, for different values of the regularity parameter $\bo$. As we can see, boxplots deteriorates only slightly for $n=100$ and $250$ when $\bo$ increases from $0.5$ to $0.8$ but
%are similar for both estimators when $\bo=0.5$ or $\bo =0.8$ but with
the dispersion for $\widehat H_n^{(2)}$ is quite larger. For $\bo=0.95$, $\widehat H_n^{(2)}$ clearly outperforms $\widehat H_n^{(1)}$ with $n=500$ observations. Estimation appears more difficult for smaller values of $n$, but it is a quite typical behavior in our considered framework.
\begin{figure}
\vspace*{0.2cm}\begin{center}
\begin{minipage}{.45\linewidth}
\includegraphics[height=4.2cm,width=5.5cm]{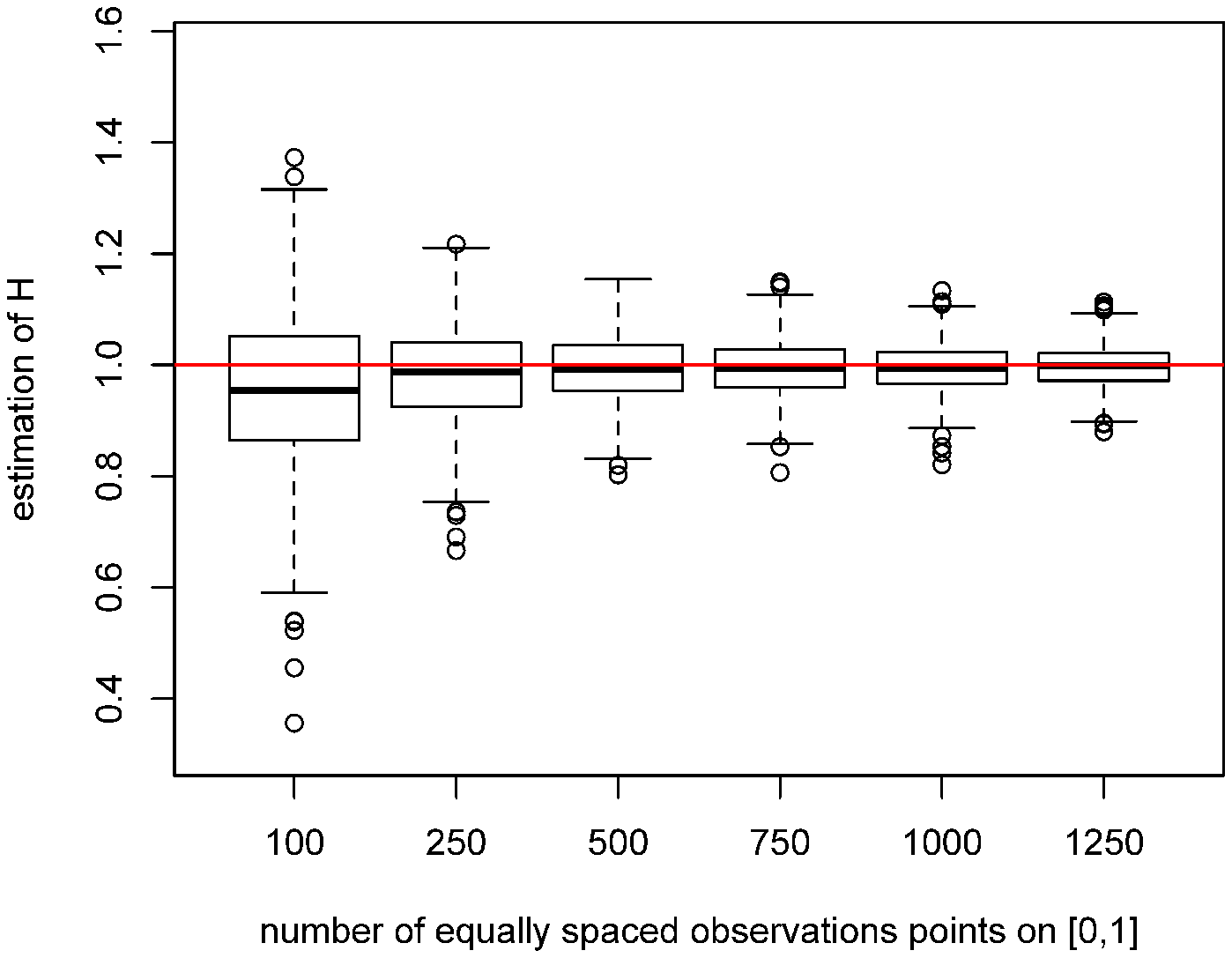}
\end{minipage}
\hspace*{0.3cm}\begin{minipage}{.45\linewidth}
{\includegraphics[height=4.2cm,width=5.5cm]{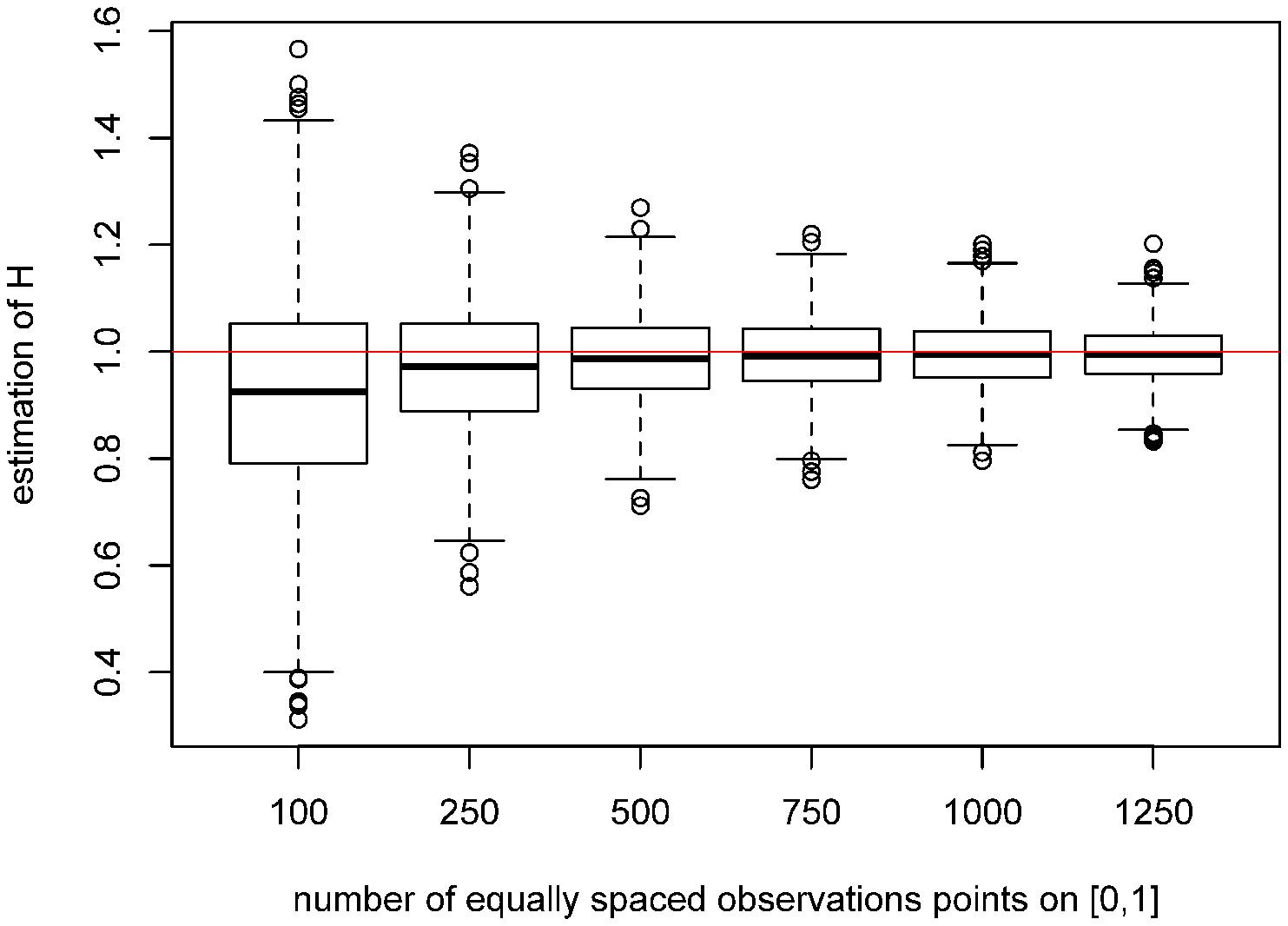}}
\end{minipage}
\\[20pt]
\begin{minipage}{.45\linewidth}
{\includegraphics[height=4.2cm,width=5.5cm]{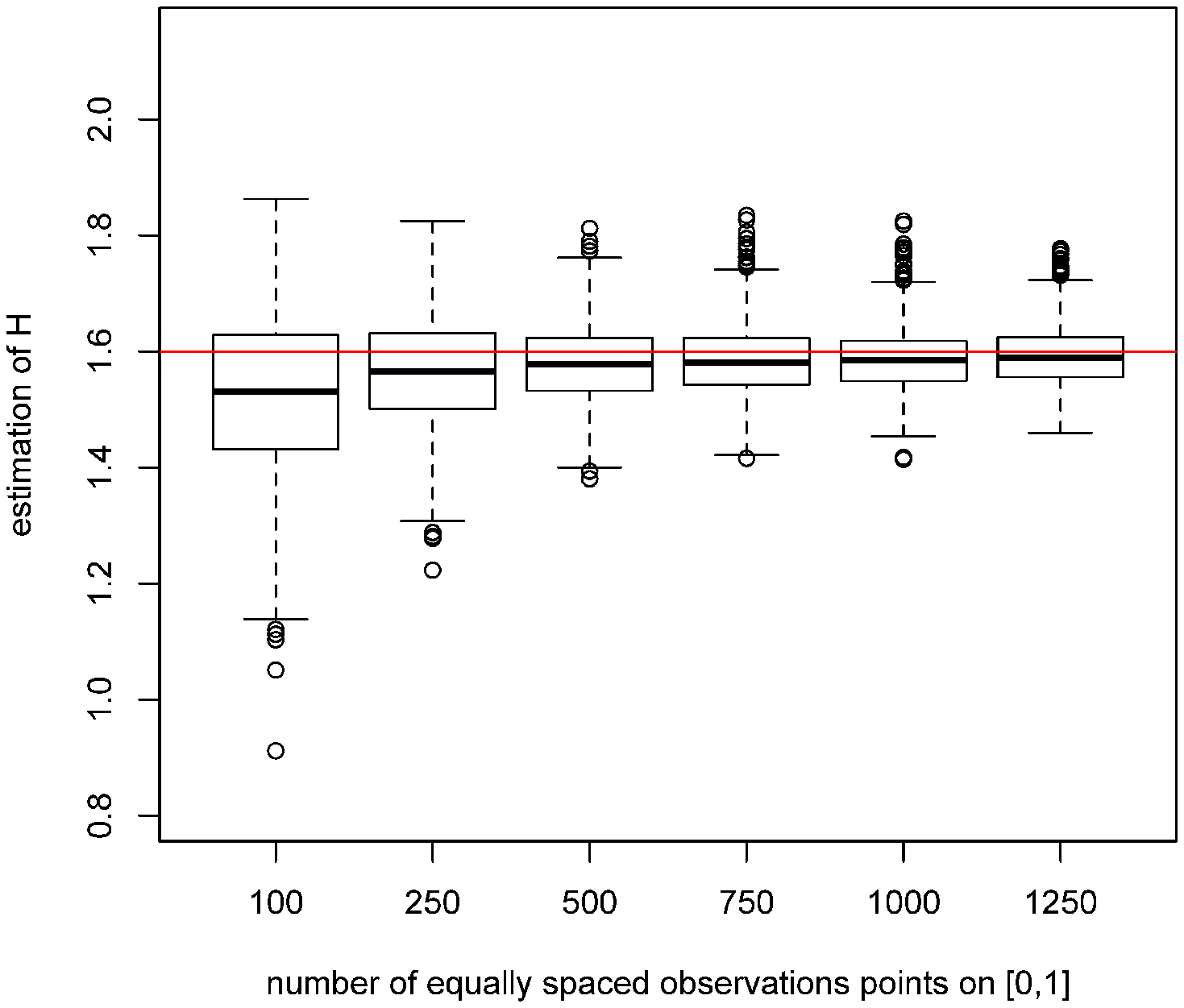}}
\end{minipage}
\hspace*{0.3cm}\begin{minipage}{.45\linewidth}
{\includegraphics[height=4.2cm,width=5.5cm]{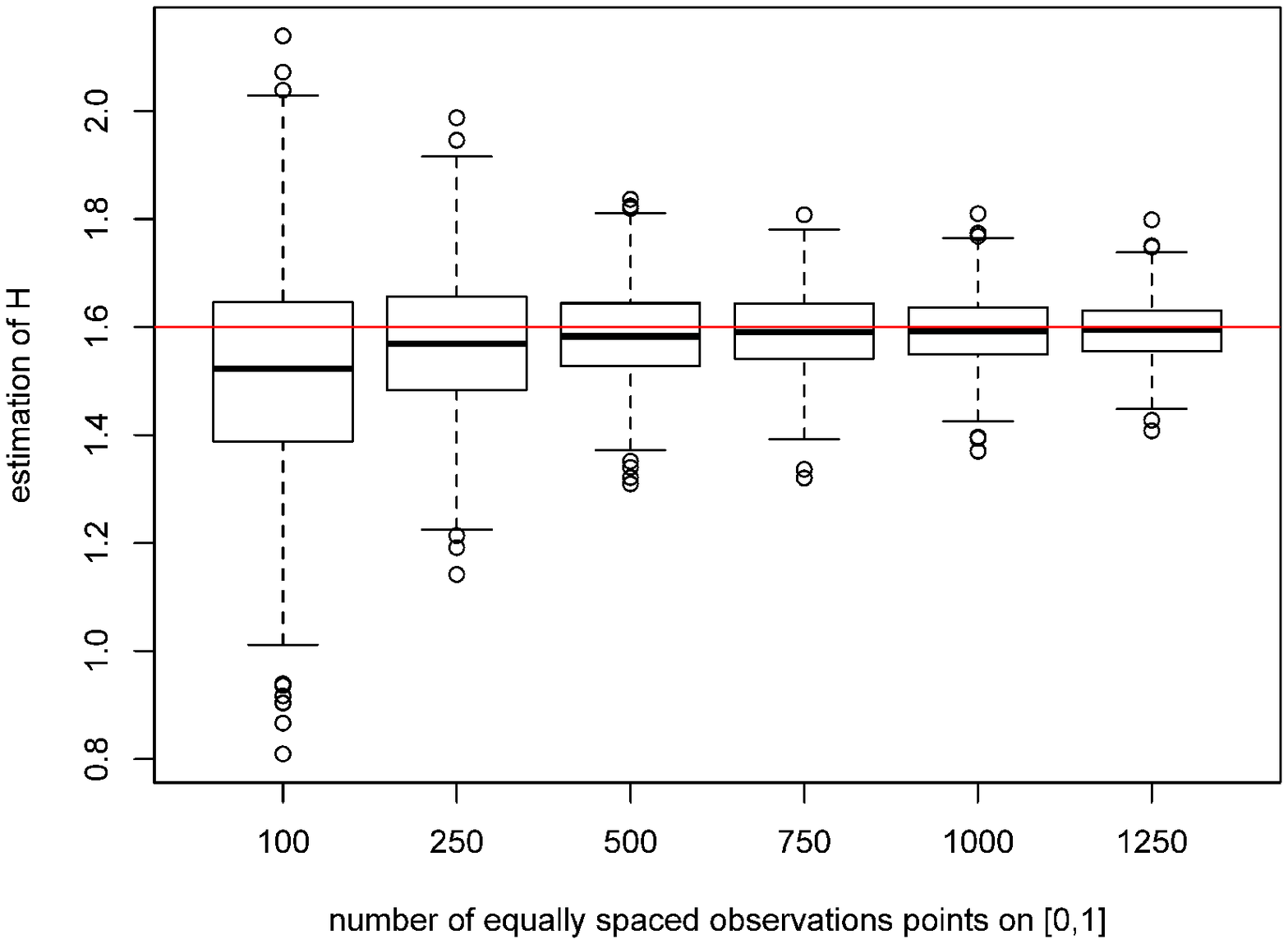}}
\end{minipage}
\\[20pt]
\begin{minipage}{.45\linewidth}
\includegraphics[height=4.2cm,width=5.5cm]{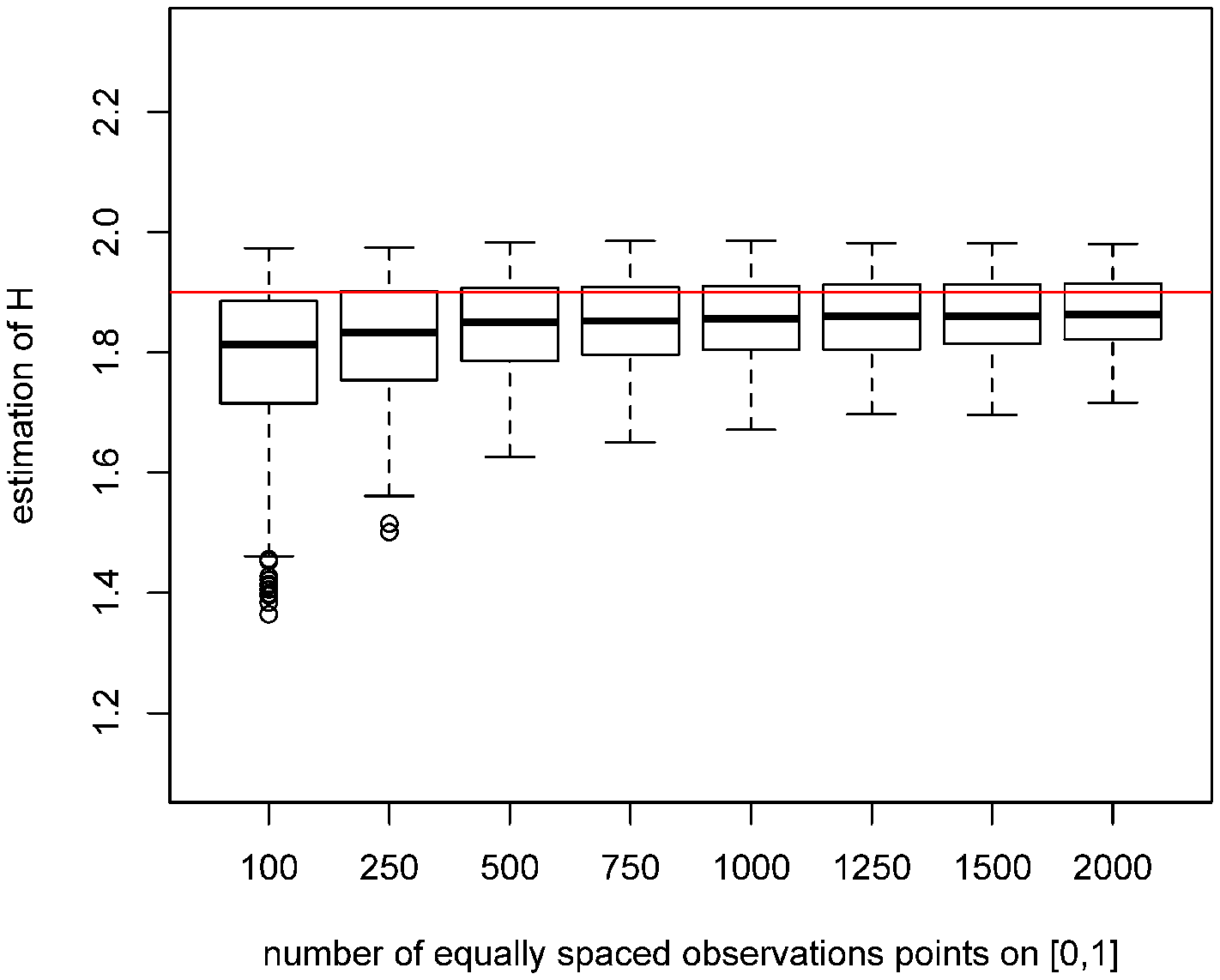}
\end{minipage}
\hspace*{0.3cm}\begin{minipage}{.45\linewidth}
{\includegraphics[height=4.2cm,width=5.5cm]{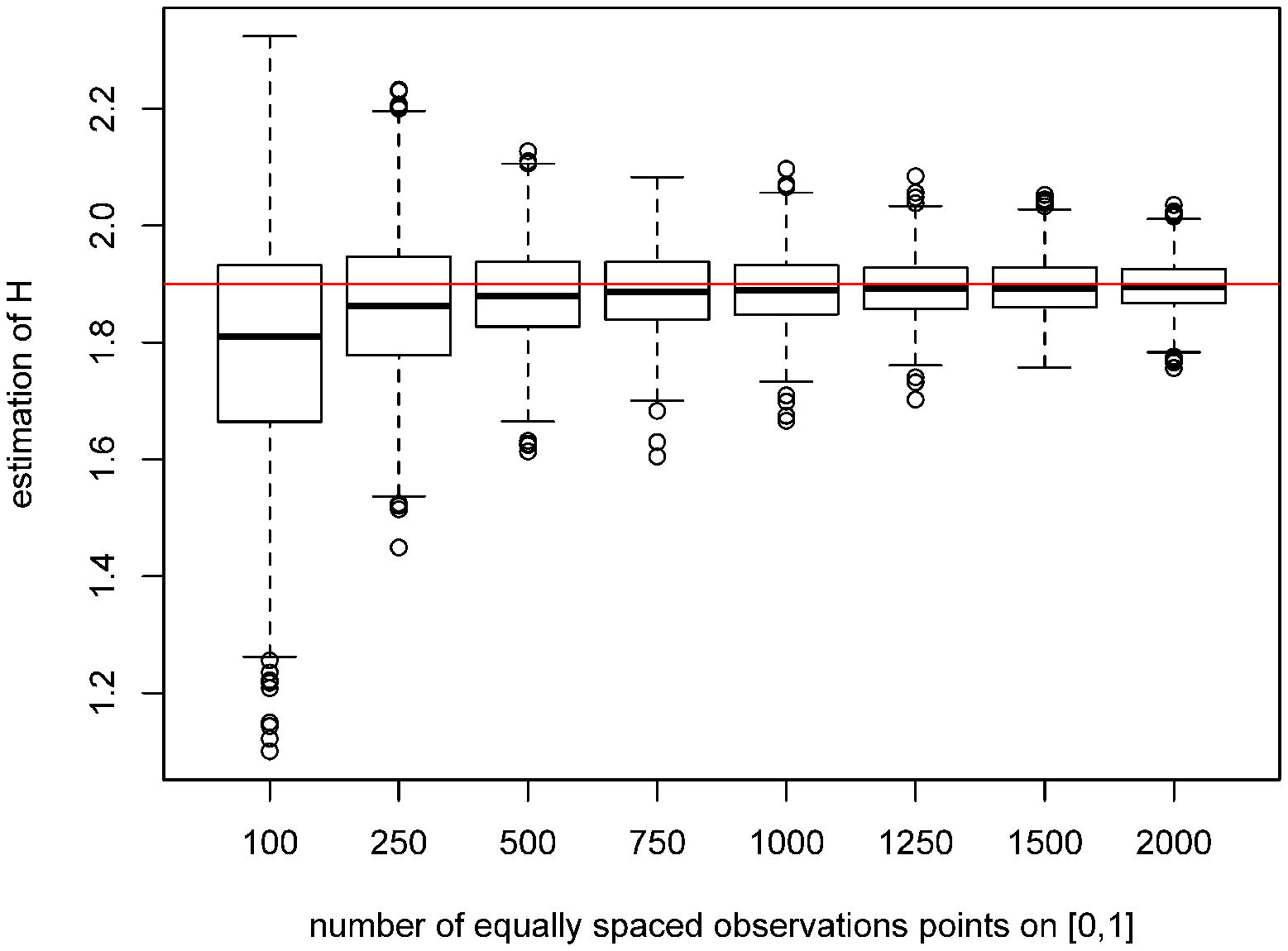}}
\end{minipage}
\vskip 5pt
$\widehat{H}^{(1)}_n$\hskip 4cm$\widehat{H}^{(2)}_n$
\end{center}
\caption{Each boxplot corresponds to 1000 estimations of $H$ by $\widehat{H}^{(1)}_n$ on the left and  $\widehat{H}^{(2)}_n$ on the right of the graph. Each realization consists in $n$ equally spaced observations on $[0,1]$ of a fBm with $\bo=0.5$ (top), $\bo=0.8$ (middle), $\bo=0.95$ (bottom), where $n=100,250,500,750,1000,1250,1500,2000$. The solid line corresponds to the real value of $H$.}\label{figboxplot}
\end{figure}

\begin{table}
\caption{{\small Rates of convergence illustrated by linear regression for $n$ in $\{500,750,1000,1250\}$.}
\label{tabvitesse}}

\medskip

\centering
{\begin{tabular}{|r|cc|cc|}
\cline{2-5}
\multicolumn{1}{c|}{}&\multicolumn{2}{|c|}{$\widehat{H}^{(1)}_n$}&\multicolumn{2}{|c|}{$\widehat{H}^{(2)}_n$}\\\cline{2-5}
\multicolumn{1}{c|}{}&slope&$R^2$&slope&$R^2$\\
\hline
fBm $\bo=0.5$&-0.488&0.998&-0.489&0.995\\
 $\bo=0.6$&-0.475&0.998&-0.488&0.995\\
 $\bo=0.7$&-0.426&0.994&-0.489&0.997\\
%$\bo=0.7$&-0.522&0.995&-0.536&0.999\\
$\bo=0.8$&-0.334&0.989&-0.491&0.997\\
%$\bo=0.8$&-0.42405&0.9922&-0.532185&0.9989\\
$\bo=0.9$&-0.225&0.990&-0.495&0.997\\
%$\bo=0.9$&-0.24223&0.991&-0.49767&0.9987\\
$\bo=0.95$&-0.186&0.995&-0.503&0.997\\
%$\bo=0.95$&-0.199230&0.9931&-0.501993&0.9986\\
\hline
ifBm $\bo=0.9$&-0.302&0.987&-0.561&0.999\\
 $\bo=0.95$&-0.244&0.978&-0.559&0.999\\
\hline
\end{tabular}
}
\end{table}

\subsubsection{Impact of misspecification of regularity}

Next, Table~\ref{tabHro} illustrates the impact of estimating $H$ when  the order $r$ in quadratic variation is misspecified. In fact estimating $\bo$ requires the knowledge of $\ro$ or an upper bound of it.  On the other hand, working with a too high value of $\ro$ may induce artificial variability in estimation, so a precise estimation of $\ro$ is important. Here, our numerical results show that, if the order $r$ of quadratic variation used for estimating $\bo$ is less than $\ro+1$, then the quantity estimated is $2r$ and not $2(r+\bo)$.

\begin{table}
\caption{{\small Mean value and standard deviation (between brackets) of the estimator $\widehat H^{(p)}_n$ based on quadratic variation of order $\ro$ or $\ro-1$ instead of $\widehat{\ro}+1$  or $\widehat{\ro}+2$. \label{tabHro}}}

\medskip

\centering
{\begin{tabular}{|r|cc|cc|}
\cline{2-5}
\multicolumn{1}{c|}{}&\multicolumn{2}{|c|}{Order $\ro-1$}&\multicolumn{2}{|c|}{Order $\ro$}\\\cline{2-5}
\multicolumn{1}{c|}{}&\multicolumn{4}{|c|}{number $n$ of equidistant observations}\\
\multicolumn{1}{c|}{}&100&500&100&500\\
\hline
ifBm $\bo=0.2$&&&\tta{1.903}{0.064}&\tta{1.961}{0.025}\\
 $\bo=0.5$&&&\tta{1.955}{0.035}&\tta{1.992}{0.006}\\
$\bo=0.8$&&&\tta{1.966}{0.024}&\tta{1.994}{0.004}\\
\hline
\tta{\text{CARMA(3,1)}}{\ro=2,\bo=0.5}&\tta{1.970}{0.0140}&\tta{1.994}{0.003}&\tta{3.919}{0.058}&\tta{3.985}{0.0109}\\
\hline
\end{tabular}
}
\end{table}

\subsubsection{Processes with varying trend or non constant function $d_0$}

{\begin{figure}
\vspace*{0.2cm}\begin{center}\includegraphics[width=5cm]{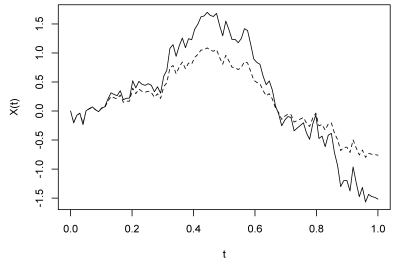}
\end{center}
\caption{Wiener process (solid) and its locally stationary transformation (dashed) used in Table~\ref{tabSelenjev}.\label{fig32}}
\end{figure}}

All previous examples are locally stationary with a constant function $d_0$. Processes meeting our conditions but with no stationary increments may be constructed using Lemma~\ref{l61}. As an example, from $Y$ a standard Wiener process ($\ro=0$, $\bo = 0.5$) or an integrated one ($\ro=1$, $\bo = 0.5$), we simulate $X(t)= (t^{\ro+0.7}+1)\,Y(t)$ having the regularity $(\ro,0.5)$ and $d_0(t)$ equaling to $(t^{\ro+0.7}+1)^2$. Figure~\ref{fig32} illustrates a Wiener sample path and its transformation.  Results are summarized in Table~\ref{tabSelenjev}:~comparing with Table \ref{tabHp} ($\bo=0.5$), it appears that  the estimation is only slightly damaged for  $\ro=1$ but of the same order when $\ro=0$. Other non stationary  processes may also be obtained by adding some smooth trend.  To this aim, we used same sample paths as for Table~\ref{tabHp}  with the additional trend $m(t)=(1+t)^2$, see Figure~\ref{Figtrend}. We may noticed in Table~\ref{tabtrend} that we obtain exactly the same results for the estimator $\widehat H_n^{(2)}$ and that only a slight loss is observed for $\widehat H_n^{(1)}$.

\begin{figure}
\begin{center}
\includegraphics[width=5cm]{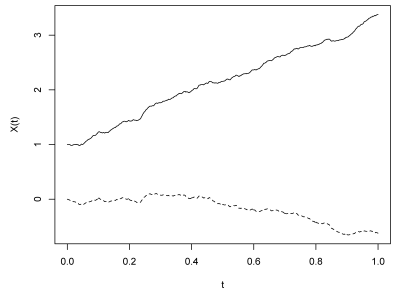}
\end{center}
\caption{Sample path of a fBm with $\bo=0.8$ (dashed line) and the same with a trend $m(t)=(1+t)^2$ (solid line).\label{Figtrend}}
\end{figure}

\begin{table}
\caption{{\small Value of MSE and bias (between brackets)
 for non constant $d_0(\cdot)$. \label{tabSelenjev}}}

\medskip

\centering
\begin{tabular}{|c|c|c|}
\cline{2-3}
\multicolumn{1}{c|}{}&&\\[-5pt]
\multicolumn{1}{c|}{}&Wiener & Integrated Wiener\\\hline
$\widehat H_n^{(1)}$&\tta{0.0021}{-0.0032}&\tta{0.0032}{0.0061}\\\hline
$\widehat H_n^{(2)}$&\tta{0.0043}{-0.0042}&\tta{0.0058}{-0.0091}\\\hline
\hline
\end{tabular}
\end{table}

\subsection{Real data}
Let us turn to examples based on real data sets. In this part, we compare our estimators of $H$ with those proposed by \citet{CH94,KW97}. We compute estimated values by setting $(u,v)=(1,m)$ in \eqref{e51} with $m$ in $\{2,4,6,8,10\}$ while for $\widehat{\alpha}_{OLS}^{(p)}$, $p=1,2$, defined in Remark~\ref{RqKW} regression is carried out over $\vu=(\ln(u),u=1,\ldots, m)\T$.
\subsubsection{Roller data}
We first focus on roller height data introduced by~\citet{La94}, which consists in $n=1150$ heights measure at 1 micron intervals along a drum of a roller. This example was already studied in~\citet{KW97}: they noticed that local self similarity  may hold at sufficiently fine scales, so  the regularity $\ro$ was supposed to be zero. Indeed, our estimator $\widehat \ro$, directly used on the data with $b_n=1/\ln(n)$, gives $\widehat\ro=0$  (with a value of $n^{4-2}\overline{\big( \Delta_2 ^{(1)}X\big)^2}$ equal to $1172345$). Next, we compute the values obtained for the estimation of $H$ in Table~\ref{tabroller}, where values of estimates proposed by \citet{CH94,KW97} are also reported for comparison. %These estimates are obtained by considering ordinary least squares estimators $\widehat{\alpha}^{(p)}_{\text{OLS}}$, with $p$ adapted to the regularity of the process (supposed to be known in their works) and  $m$ represents the number of covariates. Note that, with Kent and Wood's notation, we have in this case $H= 2\bo = \alpha$ and for $(u,v)=(1,2)$, one gets $\widehat{H}_n^{(1)}= \widehat{\alpha}^{(0)}_\text{OLS}$ and $\widehat{H}_n^{(2)}= \widehat{\alpha}^{(1)}_\text{OLS}$ but $\widehat{H}$ and $\widehat{\alpha}$ differ for other values of $v$.
It should be observed that our simplified estimators present a similar sensitivity to the choice of $m$.

\begin{table}[t]
\caption{{\small Value of MSE and bias (between brackets)  for estimators $\widehat H_n^{(p)}$, for $p=1$ or $2$ in presence of a smooth trend.
 \label{tabtrend}}}

\medskip

\centering
{\begin{tabular}{|c|cc|cc|}
\cline{2-5}
\multicolumn{1}{c|}{}&&&&\\[-5pt]
\multicolumn{1}{c|}{}&\multicolumn{2}{c|}{fBm} &\multicolumn{2}{c|}{ifBm}\\
\multicolumn{1}{c|}{}&$\bo=0.5$&$\bo=0.8$&$\bo=0.5$&$\bo=0.8$\\
\hline
%$\widehat H_n^{(1)}$&\tta{0.0208}{0.1380}&\tta{0.1170}{0.3420}&\tta{0.9444}{0.9718}&\tta{}{}\\\hline
%$\widehat H_n^{(2)}$&\tta{0.0039}{-0.0035}&\tta{0.0039}{0.0101}&\tta{0.0103}{0.0781}&\tta{}{}\\\hline
%\hline
$\widehat H_n^{(1)}$&\tta{0.0022}{0.0151}&\tta{0.0283}{0.1525}&\tta{0.0027}{0.0083}&\tta{0.0141}{0.0863}\\\hline
$\widehat H_n^{(2)}$&\tta{0.0039}{-0.0038}&\tta{0.0040}{-0.0057}&\tta{0.0051}{-0.0060}&\tta{0.0046}{-0.0060}\\\hline
\end{tabular}
}
\end{table}

\begin{table}[b]
\caption{Estimates in the roller height example
 \label{tabroller}}
\centering
\begin{tabular}{|c|cc|cc|}
\hline
$m$ & $\widehat{H}_n^{(1)}$ & $\widehat{\alpha}^{(0)}_{\text{OLS}}$ & $\widehat{H}_n^{(2)}$& $\widehat{\alpha}^{(1)}_{\text{OLS}}$\\
\hline
2 &0.63 &  0.63 & 0.77 & 0.77\\
4 &0.50 &  0.51 &0.63  & 0.65\\
6 & 0.38 & 0.39 &0.49  &  0.51\\
8 & 0.35 &  0.33 & 0.44 & 0.42\\
10 & 0.30 & 0.28 &0.39  &  0.35\\
\hline
\end{tabular}
\end{table}

\subsubsection{Biscuit data}

\begin{figure}[t]
\vspace*{0.2cm}
\begin{center}
\begin{minipage}{.44\linewidth}
\includegraphics[height=3.6cm]{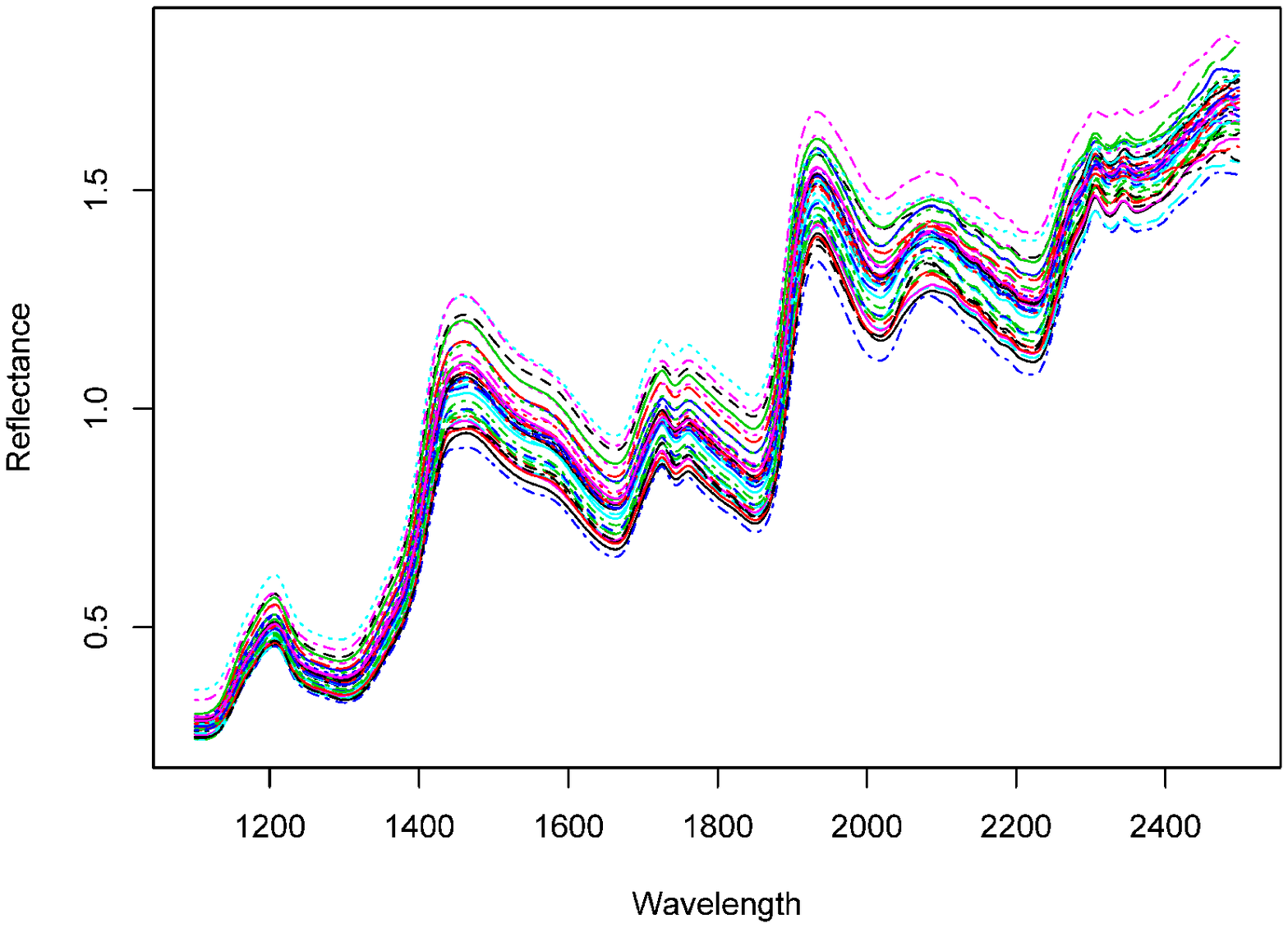}
\end{minipage}
\begin{minipage}{.45\linewidth}
\hspace*{0.2cm}\includegraphics[height=3.8cm]{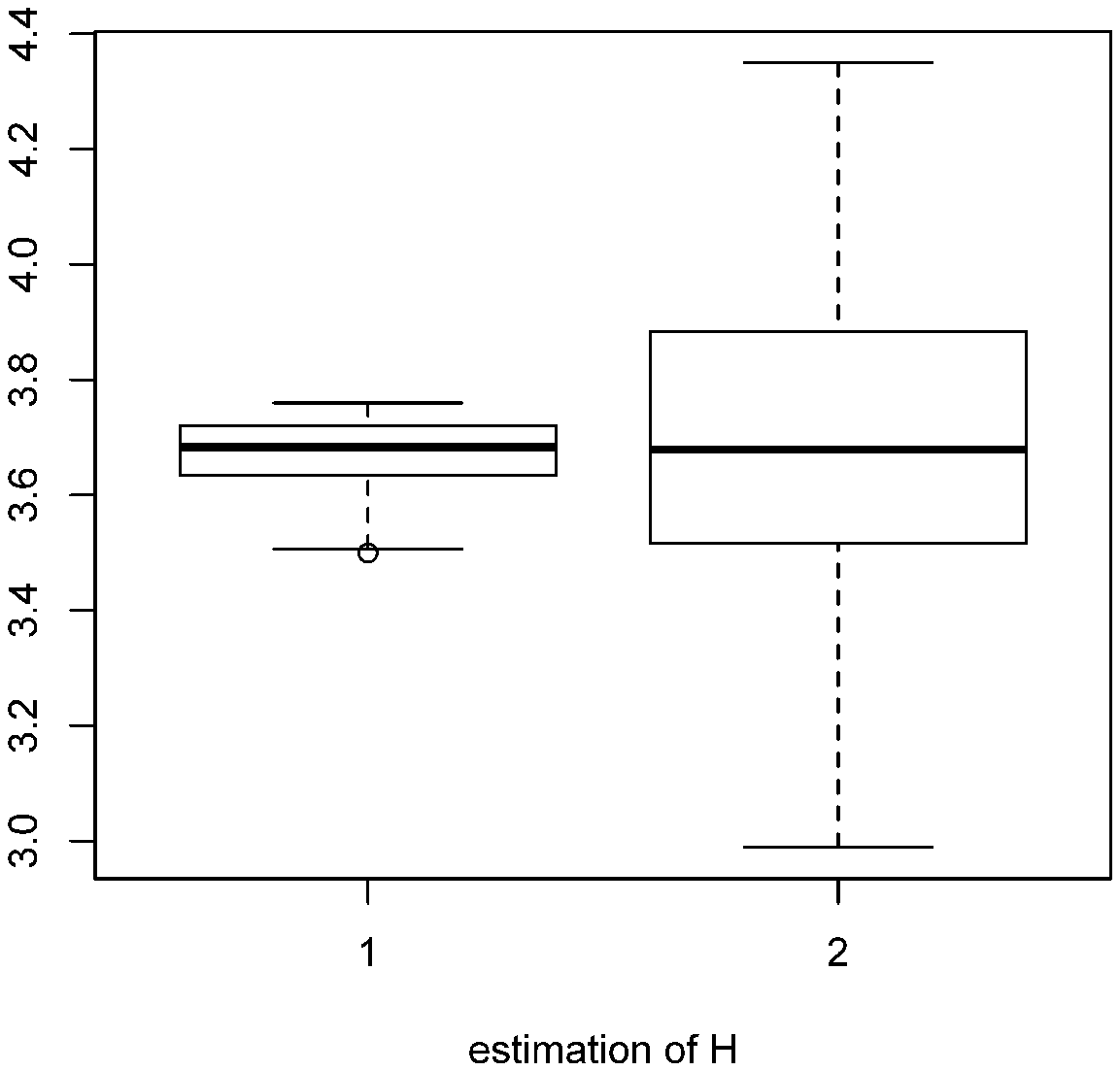}
\end{minipage}
\vskip 5pt
(a)\hskip5cm (b)
\end{center}
\caption{(a) Curve drawing reflectance in function of wavelength, varying between $1100$ and $2498$. (b) Box-plots for both estimators on the left $\widehat H_n^{(1)}$, on the right $\widehat H_n^{(2)}$ for the $39$ curves and $(u,v)=(1,4)$.}\label{figbiscuit}
\end{figure}

\begin{table}[b]
\caption{Means of estimates in the biscuit example
 \label{tab59}}
\centering
\begin{tabular}{|l|ccccc|}
\cline{2-6}
\multicolumn{1}{c|}{}&$m=2$&$m=4$&$m=6$ & $m=8$ & $m=10$\\
\hline
$\widehat{H}_n^{(1)}$& 3.60 (0.12)& 3.67 (0.07) & 3.65 (0.05) &3.62 (0.04) &3.59 (0.04)\\
$\alpha^{(1)}_{\text{OLS}}$ & 3.60 (0.12)& 3.67 (0.07) &  3.66 (0.05)& 3.63 (0.04) & 3.60 (0.03)\\
\hline
$\widehat{H}_n^{(2)}$&  2.84 (0.45) & 3.69 (0.30) &  3.83 (0.24) & 3.84 (0.19) & 3.83 (0.16)\\
$\alpha^{(2)}_{\text{OLS}}$ & 2.84 (0.45) &  3.67 (0.31) & 3.91 (0.23) & 3.98 (0.18) & 3.99 (0.14)\\
\hline
\end{tabular}
\end{table}

Now, in order to compare the (empirical) variances of these estimators, we consider a  second example introduced by \citet{BFV01}.
The experiment involved varying the composition of biscuit dough pieces and data consist in  near infrared reflectance (NIR) spectra for the same
dough. The 40 curves are graphed on the figure~\ref{figbiscuit}.  Each represents the near-infrared spectrum reflectance measure at each $2$ nanometers from $1100$ to $2498$ nm, then $700$ observation points for each biscuit. According to \citet{BFV01}, the observation 23 appears as an outlier. We estimate $\ro$ for each of the left 39  curves, using the threshold $b_n=1$, which gives $\widehat\ro=1$ for each curve. Furthermore, the averaged mean quadratic variation $n^{2r-2}\overline{\big( D_r^{(1)}X\big)^2}$ equals  to $0.33$ when $r=2$ and $122133$ when $r=3$, this explosion confirming the choice $\widehat{\ro}=3-2=1$. We turn to estimation of $H$, having in mind the comparison of our estimators together with $\alpha_\text{OLS}^{(p)}$ (where $p=1$ corresponds to the choice $(1,-2,1)$ for $a_{jr}$ and $p=2$ to the choice $(-1,3,-3,1)$). The results are summarized in Table~\ref{tab59} where it appears that, for order $\widehat{\ro}+2=3$, our estimator $\widehat{H}_n^{(2)}$ seems to be less sensitive toward high values of $m$. Also our simplified estimators present a similar variance to  $\widehat{\alpha}_{OLS}^{(p)}$, $p=1,2$. To conclude this part, it should be noticed that for the 23rd curve, the choice $m=4$ gives $\widehat{H}_n^{(1)} = 3.64$ and $\widehat{H}_n^{(2)} = 3.55$. It appears that, in both cases,  these values belong to the interquartile range obtained from the 39 curves, so at least concerning the regularity, the curve 23 should not be considered as an outlier.

\section{Annexes}

\subsection{Proofs of section \ref{Framework}}

\begin{lma}\label{l61}
Let $Y$ be a zero mean process with given regularity $(\ro,\bo)$ and asymptotic function $d_0(t)\equiv
C_{\ro,\bo}$ that satisfies A\ref{h21}(iii-$p$) ($p=1$ or 2). For a positive function $a\in C^{\ro+p}([0,T])$ and   $m\in C^{\ro+p}([0,T])$, if $X(t)=a(t)Y(t)+m(t)$, then $X$ has regularity $(\ro,\bo)$ with asymptotical function $D_{\ro,\bo}(t)
=a^2(t) C_{\ro,\bo}$ and satisfies A\ref{h21}(iii-$p$).
\end{lma}

\begin{proof} See \citet{Se00} and straightforward computation.
\end{proof}

\begin{lma}\label{l62}
Under Assumptions A\ref{h22}, we get for  $k=0,\dotsc,n-1$, and $i=1,\dotsc,n-k$:
\begin{align}
t_{k+i} - t_k &\ge C_1 i \dn, \;\; C_1 = (T\sup_{t\in[0,T]} \psi(t))^{-1},
\label{e61}\\
\label{e62} t_{k+i} - t_k &\le C_2 i\dn, \;\; C_2 = (T\inf_{t\in[0,T]} \psi(t))^{-1},\\
\intertext{and if $i=1,\dotsc, i_{\max}$ with $i_{\max}$ not depending on $n$:}
t_{k+i} - t_k &= \frac{i\dn}{T\psi(t_k)} (1+ {\cal O} (\dn^{\alpha}))
\label{e63}
\end{align}
where ${\cal O}(\dots)$ is uniform over $i$ and $k$.
\end{lma}

%\textit{Proof of Lemma~\ref{l62}}
\begin{proof}
Relations \eqref{e61}-\eqref{e62} are obtained with the mean-value theorem that induces, for $k=0,\dotsc,n-1$ and $i=1,\dotsc,n-k$:
$$ t_{k+i} -t_k = \frac{i\dn}{T\psi(t_k + \theta (t_{k+i} - t_k))}, \;\; 0<\theta <1.$$
To obtain the equivalence \eqref{e63}, one may write $t_{k+i} - t_k = \frac{i\dn}{T\psi(t_k)} (1+ R_n)$ with $R_n$ defined by
$$
R_n = \frac{\psi(t_k) - \psi(t_{k} + \theta (t_{k+i} - t_k))}{\psi(t_k + \theta (t_{k+i} - t_k))} \le \frac{L \abs{t_{k+i} - t_k}^{\alpha}}{\inf_{t\in [0,T]} \psi(t)} = {\cal O}(\dn^{\alpha})
$$
by Assumption A\ref{h22} and uniformly over $i$, $n$ and $k$ for $i=1,\dotsc,i_{\max}$.
\end{proof}
\begin{lma}\label{l63}
We have under Assumption~A\ref{h22} and for $r\ge 1$, $i=0,\dotsc,r$,
\begin{enumerate}
\item for $p=0,\dotsc,r-1$ and convention $0^0=1$
\begin{equation} \label{e64} \sum_{i=0}^r (t_{k+iu} - t_k)^p b_{ikr}^{(u)} = 0, \end{equation}
\item \begin{equation} \label{e65}  \sum_{i=0}^r (t_{k+iu} - t_k)^r b_{ikr}^{(u)} = 1,  \end{equation}
\item  \begin{equation} \label{e66} \abs{b_{ikr}^{(u)}} \le \frac{C_1^{-r}u^{-r}\dn^{-r},}{\prod_{m=0,m\not=i}^r  \abs{i-m}}, \end{equation}
with $C_1$ given by \eqref{e61},
\item \begin{equation} b_{ikr}^{(u)} = \frac{u^{-r} \psi^{r}(t_k)T^r\dn^{-r} }{\prod_{m=0,m\not=i}^r (i-m)}(1+ {\cal O}(\dn^{\alpha}))\label{e67} \end{equation}
with ${\cal O}(\dots)$ uniform over $i$ and $k$. \end{enumerate}
\end{lma}
%\textit{Proof of Lemma~\ref{l63}}
\begin{proof}
The term  $g[t_k,\dotsc,t_{k+ru}]= \sum_{i=0}^r b_{ikr}^{(u)} g(t_{k+iu})$ is the leading coefficient in the polynomial approximation of degree $r$ of $g$, given in the decomposition \eqref{e22}. Considering the polynomial  $g(t) = (t-t_{k})^p$,   we may immediately deduce  the properties \eqref{e64}-\eqref{e65}, from  uniqueness of relation \eqref{e22}. Next, \eqref{e66}-\eqref{e67} are direct consequences of Lemma~\ref{l62} and  definition \eqref{e23} of $b_{ikr}^{(u)}$.
\end{proof}

\noindent \begin{pro}\label{p61} Under Assumption~A\ref{h21} and A\ref{h22}, one obtains: \\
\indent (i)  for  $r=\ro+p$ with $p=1,2$:
$$
n^{-2(p-\bo)}\,\esp\Big(\,\overline{\big( D_{\ro+p}^{(u)}X\big)^2}\,\Big) \tv[n\to\infty ]{ }u^{-2(p-\bo)}\ell(p,\ro,\bo)
$$
where $\dsp \ell (p,0,\bo) = -\frac{1}{2} \int_0^T d_0(t) \frac{\psi^{2p+1}(t)}{\psi^{2\bo}(t)} \dt\sum_{i,j=0}^{p} \frac{\abs{i-j}^{2\bo}}{\prod_{\substack{m=0\\m\not=i}}^{p} (i-m) \prod_{\substack{q=0\\q\not=j}}^{p} (j-q)} $
while if $\ro\ge 1$,
\begin{equation}
\label{e68} \ell (p,\ro,\bo) =\frac{(-1)^{\ro+1} \int_0^T d_0(t) \frac{\psi^{2p+1}(t)}{\psi^{2\bo}(t)}\dt}{2(2\bo+2\ro)\dotsm(2\bo+1)} \sum_{i,j=0}^{\ro+p}  \frac{\abs{i-j}^{2(\ro+\bo)}}{\prod_{\substack{m=0\\m\not=i}}^{\ro+p} (i-m) \prod_{\substack{q=0\\q\not=j}}^{\ro+p} (j-q) }
\end{equation}

(ii) for $\ro \ge 1$ and $r=1,\dotsc,\ro$: $$
\esp\Big(\,\overline{\big( D_r^{(u)}X\big)^2}\,\Big) \tv[n\to\infty ]{ } \frac{1}{(r!)^2} \int_0^T \esp \big(X^{(r)}(t)\big)^2 \psi(t)\dt.$$
\end{pro}

%\textit{Proof of Proposition~\ref{p61}}
\begin{proof}
{\bf A.} Let us begin with general expressions of $\dsp\esp\Big(\,D_{r,k}^{(u)}\,X D_{r,\ell}^{(u)}\,X\Big)$ useful for the sequel. First for $\mathds{L}^{(p,p)}(s,t) = \esp \big(X^{(p)}(s)X^{(p)}(t)\big)$ ($p\ge 0$), the relation \eqref{e21} is equivalent to
\begin{equation} \label{e69}
\lim_{h\to 0} \sup_{\substack{s,t\in[0,T]\\ \abs{s-t} \le h, s\not=t}} \abs{\frac{\ll^{(\ro,\ro)}(s,s) +
\ll^{(\ro,\ro)}
(t,t) - 2 \ll^{(\ro,\ro)}(s,t)}{\abs{s-t}^{2\bo}} - d_0(t)} = 0.
\end{equation}
For $(v,w)\in[0,1]^2$, we set  $\dot{v}_{ik} = t_k +(t_{k+iu} - t_k)v$ and $\dot{w}_{j\ell} = t_{\ell} + (t_{\ell+ju} - t_{\ell})w$.
Next, from the definition of $D_{r,k}^{(u)}\,X$ given in \eqref{e24}, we get
$$ \esp(D_{r,k}^{(u)}\,X D_{r,\ell}^{(u)}\,X) = \sum\limits_{i,j=0}^{r} b_{ikr}^{(u)}b_{j\ell r}^{(u)}
\ll^{(0,0)}(t_{k+iu},t_{\ell+ju}).$$
For $\ro=0$ and since $\sum_{i=0}^{r} b_{ikr}^{(u)}=0$, we have:
\begin{multline} \label{e610} \esp(D_{r,k}^{(u)}\,X D_{r,\ell}^{(u)}\,X) = \sum_{i,j=0}^{r}
\!b_{ikr}^{(u)}b_{j\ell r}^{(u)}
\Big\{ \ll^{(0,0)}(t_{k+iu},t_{\ell+ju})\\ -\frac12 \ll^{(0,0)}(t_{k+iu},t_{k+iu}) -\frac12
\ll^{(0,0)}(t_{\ell+ju},t_{\ell+ju}) \Big\}.
\end{multline}

If $\ro\ge 1$, we apply multiple Taylor series expansions with integral remainder. Next,  the properties $\sum_{i=0}^{r}
b_{ikr}^{(u)} (t_{k+iu} - t_k)^p=0$ for $p=0,\dotsc, r-1$   (and convention $0^0=1$) induce :
\begin{multline}
\label{e611} \esp(D_{r,k}^{(u)}\,X D_{r,\ell}^{(u)}\,X) =  \sum_{i,j=0}^{r} \!b_{ikr}^{(u)}b_{j\ell r}^{(u)}
(t_{k+iu} - t_k)^{\ra}(t_{\ell+ju} - t_\ell)^{\ra}\\ \times  \intd_{[0,1]^2} \!\frac{\!(1-v)^{\ra-1}\!(1-w)^{\ra-1}}{((\ra -
1)!)^2} \ll^{(\ra,\ra)}(\dot{v}_{ik},\dot{w}_{j\ell}) \dv\!\dw
\end{multline}
where we have set $\ra=\min(\ro,r)\ge 1$.

\medskip\

\noindent {\bf B.} From expressions \eqref{e610}-\eqref{e611}, we are in position to derive the asymptotic behavior of $ \esp\big(\,\overline{(D_r^{(u)}\, X)^2}\big)$.

\paragraph{\emph{Case $\ro \ge  1$, $r=\ro+p$, $p=1$ or $p=2$.}}

In this case, $\ra=\ro\le r-1$. From \eqref{e611} and the property $\sum_{i=0}^{r} b_{ikr}^{(u)} (t_{k+iu} - t_k)^{\ro}=0$, we may
write
%{\small
\begin{multline} \label{e612} \esp(D_{r,k}^{(u)}\,X)^2 =  \sum_{i,j=0}^{r} b_{ikr}^{(u)}b_{jkr}^{(u)}
(t_{k+iu} - t_k)^{\ro} (t_{k+ju} - t_k)^{\ro}\!\!\int_0^1\!\!\!\int_0^1 \!\! \frac{((1-v)(1-w))^{\ro-1}}{((\ro - 1)!)^2}
 \\  \times\Big\{
\ll^{(\ro,\ro)}(\dot{v}_{ik},\dot{w}_{jk}) -\frac12 \ll^{(\ro,\ro)}(\dot{v}_{ik},\dot{v}_{ik}) -\frac12
\ll^{(\ro,\ro)}(\dot{w}_{jk},\dot{w}_{jk}) \Big\} \dv\!\dw
\end{multline}
Using the locally stationary condition \eqref{e69}, uniform continuity of $d_0(\cdot)$ on $[0,T]$ and the bound:
  $\abs{\dot{v}_{ik} - \dot{w}_{jk}}\le t_{k+ru} - t_k \le C_1 ru \dn $ for $i=0,\dotsc,r$ and $j=0,\dotsc,r$, we may show that the predominant term for $ \esp\big(\,\overline{(D_r^{(u)}\, X)^2}\big)$ is given by:
\begin{multline} \label{e613} \frac{-1}{2(n_r+1)} \sum_{k=0}^{n_r}\sum_{i,j=0}^{r} b_{ikr}^{(u)}b_{jkr}^{(u)}
(t_{k+iu} - t_k)^{\ro} (t_{k+ju} - t_k)^{\ro}  \\  \times  \intd_{[0,1]^2} \frac{((1-v)(1-w))^{\ro-1}}{((\ro - 1)!)^2}
\abs{\dot{v}_{ik} - \dot{w}_{jk}}^{2\bo} d_0(t_k) \dv\dw.
\end{multline}
From the equivalents \eqref{e63} and \eqref{e67}, we can write the leading term of \eqref{e613} as a Riemann sum on $t_k$ to obtain
\begin{multline*}\dn^{2p - 2\bo} \esp\big(\,\overline{(D_r^{(u)}\, X)^2}\big) \tv[n\to\infty]{} - \frac{1}{2} (\frac{u}{T})^{-2p+2\bo} \int_0^T d_0(t) \psi^{2p+1-2\bo}(t) \dt
\\\times   \sum_{i,j=0}^r \frac{(ij)^{\ro}}{\prod\limits_{\substack{m=0\\m\not=i}}^r (i-m) \prod\limits_{\substack{q=0\\q\not=j}}^r (j-q)} \intd_{[0,1]^2} \!\! \frac{((1-v)(1-w))^{\ro-1}}{((\ro - 1)!)^2} \abs{iv - jw}^{2\bo}\dv\dw.\end{multline*}
Next by performing elementary but tedious multiple integrations by parts, we arrive at the following  simpler form of $\ell (r,\ro,\bo)$
given in \eqref{e68}, for $n\dn\to T$.

\paragraph{\emph{Case $\ro=0$, $r=\ro+1,_,\ro+2$.}}

The proof is the same but starting from \eqref{e610} and $\ell=k$.

\paragraph{\emph{Case $\ro \ge 1$ and $r=1,\dotsc,\ro$.}}

In this case, $\ra = r$ and from the relation  \eqref{e611}, one gets
\begin{multline*}
\esp(D_{r,k}^{(u)}\,X)^2 = \sum\limits_{i,j=0}^{r} \!b_{ikr}^{(u)}b_{jkr}^{(u)}
(t_{k+iu}-t_k)^{r}(t_{k+ju}-t_k)^{r}\\ \times \intd_{[0,1]^2}
\frac{((1-v)(1-w))^{r-1}}{((r - 1)!)^2} \ll^{(r,r)}(\dot{v}_{ik},\dot{w}_{jk}) \dv\!\dw.
\end{multline*}
The result follows after Riemann summation with the help of uniform continuity of $\ll^{(r,r)}(\cdot,\cdot)$, $r=1,\dotsc,\ro$ and properties \eqref{e63}, \eqref{e65}.
\end{proof}

\subsection{Auxiliary results}

The following lemma gives some useful results on the asymptotic behavior of $\mathds{C}_r (k,\ell)$ and  $\mathds{C}_r^2 (k,\ell)$ with $\mathds{C}_r (k,\ell) = \cov\big(D_{r,k}^{(u)}\,X,D_{r,\ell}^{(u)}\,X\big)$  with $n_r = n-ur$ and $u$  a positive integer.
\begin{lma}\label{l64} Suppose that Assumption A\ref{h21} and A\ref{h22} are fulfilled.

(i) Under the condition A\ref{h21}-(iii-1) and for  $r=\ro+p$, $p=1$ or $p=2$,  one obtains
$$
\max_{k=0,\dotsc,n_r}  \sum_{\ell=0}^{n_r}
\abs{\mathds{C}_r (k,\ell)} = \begin{cases} {\mathcal O} ( n^{2p-2\bo}) &\text{ if } 0 < \bo <\frac{1}{2},\\
{\mathcal O} ( n^{2p-1}\ln n) &\text{  if } \bo =\frac{1}{2}, \\
 {\mathcal O} ( n^{2p-1}) &\text{ if } \frac{1}{2}<\bo<1~;\end{cases}$$ and
$$
 \sum_{k=0}^{n_r}  \sum_{\ell=0}^{n_r}
\mathds{C}_r^2 (k,\ell) = \begin{cases} {\mathcal O} ( n^{4p-4\bo+1}) &\text{ if } 0 < \bo <\frac{3}{4},\\
{\mathcal O} ( n^{4p -2}\ln n) &\text{  if } \bo =\frac{3}{4}, \\
 {\mathcal O} ( n^{4p-2}) &\text{ if } \frac{3}{4}<\bo<1.\end{cases}
$$

(ii) Under the condition A\ref{h21}-(iii-2) and for $r=\ro+2$, one obtains $$\max\limits_{k=0,\dotsc,n_r}  \sum_{\ell=0}^{n_r}
\abs{ \mathds{C}_r (k,\ell)} = {\mathcal O}(n^{4-2\bo}) \text{ and }
\sum_{k=0}^{n_r}  \sum_{\ell=0}^{n_r}
 \mathds{C}_r^2 (k,\ell) ={\mathcal O}(n^{9-4\bo}).$$

(iii)  If $r=1,\dotsc,\ro$ (with $\ro \ge 1$), then $\dsp \max_{k=0,\dotsc,n_r}  \sum_{\ell=0}^{n_r}
\abs{\mathds{C}_r(k,\ell)} = {\mathcal O}  (n) $ and $\dsp \sum_{k=0}^{n_r}  \sum_{\ell=0}^{n_r}
\mathds{C}_r^2(k,\ell) = {\mathcal O} ( n^{2}). $
\end{lma}

\begin{proof}

(i) Setting $\mu(t)=\esp\big(X(t)\big)$, $\mu$ is $\ro$-times differentiable and similarly to \eqref{e610}-\eqref{e611}, we get the expansion
\begin{multline*} \mathds{C}_r(k,\ell) =  \sum_{i,j=0}^{r} b_{ikr}^{(u)}b_{j\ell r}^{(u)}
(t_{k+iu} - t_k)^{\ro} (t_{\ell+ju} - t_{\ell})^{\ro}\\ \times \intd_{[0,1]^2}\frac{((1-v)(1-w))^{\ro-1}}{((\ro -
1)!)^2}\mathds{K}^{(\ro,\ro)}(\dot{v}_{ik},\dot{w}_{j\ell}) \dv\dw
\end{multline*}
for $\ro\ge 1$ while if  $\ro=0$, $\mathds{C}_r(k,\ell) = \sum_{i=0}^{r}\sum_{j=0}^{r} b_{ikr}^{(u)}b_{j\ell r}^{(u)}
\mathds{K}(t_{k+iu},t_{\ell+ju})$.

\paragraph{\emph{Case $r=\ro+1$ or $\ro+2$}.}

For $\ro \ge 1$, we have the bound:
$$\max_{k=0,\dotsc,n_r} \sum_{\ell=0}^{n_r} \abs{\mathds{C}_r(k,\ell)} \le U_{1n}+  U_{2n} +
 U_{3n}$$ with $U_{1n}= \max\limits_{k=ur+1,\dotsc,n_r} \sum\limits_{\ell=0}^{k-ur-1}\!\!\abs{\mathds{C}_r(k,\ell)}$,
$U_{2n}= \max\limits_{k=0,\dotsc,n-2ur-1} \sum\limits_{\ell=k+ur+1}^{n_r}\!\!\abs{\mathds{C}_r(k,\ell)}$
and $U_{3n}= \max\limits_{k=0,\dotsc,n_r}\sum\limits_{\ell=\max(0,k-ur)}^{\min(n_r,k+ur)}\abs{\mathds{C}_r(k,\ell)}$. First, consider the sum $U_{1n}+U_{2n}$ where
$\abs{k-\ell}\ge ur+1$. Since  $\sum_{i=0}^{r} b_{ikr}^{(u)} (t_{k+iu} - t_k)^{\ro}=0$ for $r=\ro+1$ or $r=\ro+2$, and $[t_k,\dot{v}_{ik}]$ is distinct from $[t_{\ell},\dot{w}_{j\ell}]$,  we get
\begin{multline} \label{e614}
\mathds{C}_r(k,\ell) =\sum_{i,j=0}^{r} b_{ikr}^{(u)}b_{jkr}^{(u)} (t_{k+iu} - t_k)^{\ro}(t_{\ell+ju} - t_{\ell})^{\ro} \\ \times
\intd_{[0,1]^2} \frac{((1-v)(1-w))^{\ro-1}}{((\ro - 1)!)^2}
\int_{t_k}^{\dot{v}_{ik}}\int_{t_{\ell}}^{\dot{w}_{j\ell}}
 \mathds{K}^{(\ro+1,\ro+1)}(s,t) \,\mathrm{d}s\!\dt\!\dv\!\dw.
\end{multline}
Condition  A\ref{h21}-(iii-1), together with the bounds \eqref{e62} and \eqref{e66},  gives a bound of ${\mathcal O} (n^{2p-2\bo} \sum_{i=1}^n i^{-2(1-\bo)})$ for
 $\abs{U_{1n} +U_{2n}}$, which is of order
$n^{2(p-\bo)}$ if  $0 < \bo < \frac{1}{2}$, $n^{2(p-\bo)} \ln n$ if  $\bo = \frac{1}{2}$ and $n^{2p -1}$ if
$\bo> \frac{1}{2}$. Next, for $ U_{3n}$ where $\abs{k-\ell}\le ur$, we obtain that $U_{3n} = {\cal O} (n^{2(p-\bo)})$ in a similar way as in the proof of Proposition~\ref{p61}, and with the help of Cauchy-Schwarz inequality to control the terms depending on $\mu^{(\ro)}(t)$.

We proceed similarly for the case $\ro=0$, starting from the definition of $\mathds{C}_r(k,\ell)$ as well as for the study of  $\sum_{k=0}^{n_r}\sum_{\ell=0}^{n_r} \mathds{C}_r^2(k,\ell)$ for which dominant terms are of order ${\cal O}(n^{1+4p - 4\bo}\sum_{i=1}^n i^{-4(1 - \bo)})$.

(ii) The condition  A\ref{h21}-(iii-2) and $r=\ro+2$ allows to transform \eqref{e614} into
\begin{multline} \label{e615}
\mathds{C}_r(k,\ell) =\sum_{i,j=0}^{r} b_{ikr}^{(u)}b_{jkr}^{(u)} (t_{k+iu} - t_k)^{\ro}(t_{\ell+ju} - t_{\ell})^{\ro}  \intd_{[0,1]^2} \frac{((1-v)(1-w))^{\ro-1}}{((\ro - 1)!)^2}\\
\times\int_{t_k}^{\dot{v}_{ik}}\int_{t_{\ell}}^{\dot{w}_{j\ell}} \int_{t_k}^t \int_{t_{\ell}}^s K^{(\ro+2,\ro+2)}(y,z) \dy\dz \mathrm{d}s\dt\dv\dw
\end{multline}
which gives that  $$ \max\limits_{k=0,\dotsc,n_r}\sum_{\ell=0}^{n_r} \abs{\mathds{C}_r(k,\ell)} = {\cal O}(n^{2(2-\bo)} \sum_{i=1}^n i^{-4+2\bo}) = {\cal O}(n^{2(2-\bo)})$$ for all $\bo\in]0,1[$ and $\ro\ge 1$. From \eqref{e615}, we also get that $\sum_{k=0}^{n_r}\sum_{\ell=0}^{n_r} \mathds{C}_r^2(k,\ell) = {\cal O}(n^{9-4\bo} \sum_{i=1}^n i^{-8+4\bo}) = {\cal O}(n^{9-4\bo})$ for all $\bo\in ]0,1[$.

(iii) Results of these part, where $\ro\ge 1$, are consequences of
\begin{multline*}
\mathds{C}_r(k,\ell) = \sum_{i,j=0}^r b_{ikr}^{(u)} b_{j\ell r}^{(u)} (t_{k+iu} - t_k)^r (t_{\ell+ju} - t_{\ell})^r \\ \times\intd_{[0,1]^2} \frac{( (1-v)(1-w))^{r-1}}{(r-1)!^2} K^{(r,r)} (\dot{v}_{ik}, \dot{w}_{jl}) \dv\dw = {\cal O} (1)
\end{multline*}
with uniform continuity of $K^{(r,r)}(\cdot,\cdot)$  for $r=1,\dotsc,\ro$ together with bounds \eqref{e62} and \eqref{e66}. \end{proof}
Next proposition gives a general exponential bound, involved in  all our results.
\begin{pro}\label{p62} Suppose that Assumption~A\ref{h21} and A\ref{h22} are fulfilled. Let $\eta_{n}(r)$ be  some given
positive sequence and $u\in\bN\as$,  then $$
\mathds{P}\Big( \abs{\overline{ (D_r^{(u)} X)^2} - \esp\Big(\,\overline{ (D_r^{(u)} X)^2 }\,\Big)} \ge
\eta_n(r)\Big)$$ is of order:
\begin{multline*}{\mathcal O}\bigg( \exp\Bigl(-C(r) n\eta_{n}(r)\times  \min\Bigl( \big(\max\limits_{0\leq k \leq
n_r}\sum\limits_{\ell=0}^{n_r}\abs{\, \mathds{C}_r(k,\ell)}\big)^{-1}, \frac{n\eta_n(r)}{\sum\limits_{k,\ell=0}^{n_r}  \mathds{C}_r^2(k,\ell)}\Bigr)\Bigr)\bigg)\\+  {\mathcal O}\bigg( \frac{v_n^{1/2}(r)}{n\eta_n(r)} \exp\Big( -
C(r)\frac{n^2\eta_n^2(r)}{v_n(r)}\Big)\bigg)
\end{multline*}
for  some positive constant   $C(r)$, not depending on $\eta_n(r)$ and
\begin{equation} \label{e616} v_n(r) := n \max\limits_{k=0,\dotsc,n_r} \big( \esp(D_{r,k}^{(u)}\,X)\big)^2 \max\limits_{k=0,\dotsc,n_r} \sum\limits_{\ell=0}^{n_r}\abs{\mathds{C}_r(k,\ell)}.
\end{equation}
\end{pro}

\begin{proof}
For all $r\ge 1$, we may bound $\mathds{P}\Big( \abs{\overline{ (D_r^{(u)} X)^2} - \esp\Big(\,\overline{ (D_r^{(u)} X)^2 }\,\Big)} \ge
\eta_n(r)\Big)$ by $S_1+S_2$ with $$S_1=  \p \Big( \abs{ \sum_{k=0}^{n_r}(D_{r,k}^{(u)}\,X - \esp(
D_{r,k}^{(u)}\,X))^2 -  \var(D_{r,k}^{(u)}\,X)} > \frac{(n_r+1)\eta_n(r)}{2}\Big)$$ and $\dsp
S_2=
 \p\Big( \abs{ \sum_{k=0}^{n_r}(\esp (D_{r,k}^{(u)}\,X))\big( D_{r,k}^{(u)}\,X - \esp (D_{r,k}^{(u)}\, X)\big) } >\frac{(n_r+1)\eta_n(r)}{4}\Big)$. First, let $\{Y_i\}_{i=1,\dotsc,d_n}$ be an orthonormal basis for the linear span of  $\{ D_{r,k}^{(u)}\,X\}_{k=0,\dotsc,n_r}$ (so that $Y_i$ are i.i.d. with density ${\mathcal N}(0,1)$). We can write
$\dsp D_{r,k}^{(u)}\,X-\esp\big(D_{r,k}^{(u)}\,X\big)=\sum_{i=1}^{d_n}d_{k,i}Y_i$ with $d_{k,i} = \cov\bigl(D_{r,k}^{(u)}\,X,Y_i\bigr)$.
Next, if $Y=(Y_1,\ldots,Y_{d_n})\T$, we obtain
$$ \sum_{k=0}^{n_r} (D_{r,k}^{(u)}\,X - \esp D_{r,k}^{(u)}\,X)^2 = \sum_{i,j=1}^{d_n} c_{i,j} Y_iY_{j} =
Y\T CY \text{ and }
\sum_{k=0}^{n_r} \var(D_{r,k}^{(u)}\,X) = \sum_{i=1}^{d_n} c_{i,i}
$$
with $\dsp c_{i,j}=\sum_{k=0}^{n_r}d_{ki}d_{kj}$. Next, for $ C=\bigl( c_{i,j}\bigr)_{\substack{i=1,\dotsc,d_n\\j=1,\dotsc,d_n}}$
and $D=\bigl(d_{k,j}\bigr)_{\substack{k=0,\ldots,n_r\\j=1,\dotsc,d_n}}$, one gets $C=D\T\,D$ where  $C$ is
a real, symmetric and positive semidefinite matrix. There exists an orthogonal matrix $P$ such that ${\rm
diag}(\lambda_1,\ldots,\lambda_{d_n})=P\T CP$,  for $\lambda_i$ eigenvalues of $C$. Then we can transform the
quadratic form as:
$$
\sum_{k=0}^{n_r} (D_{r,k}^{(u)}\,X - \esp(D_{r,k}^{(u)}\,X))^2  =\sum_{i=1}^{d_n}\lambda_i(P\T Y)_i^2
$$
where $(P\T Y)_i$ denotes the $i$-th component of the  ($d_n\times 1$) vector $P\T Y$. Since
$\sum_{i=1}^{d_n} c_{i,i} = \sum_{i=1}^{d_n} \lambda_i$, we arrive
at
$$ S_1= \mathds{P}\Bigl( \,\Bigl\vert \,\sum_{i=1}^{d_n} \lambda_i \bigl( (P\T Y)_i^2 -1\bigr)
\Bigr\vert\ge \frac{(n_r+1)\eta_{n}(r)}{2}
\Bigr).$$
Now, with the exponential bound of \citet{HW71}, we obtain for some generic constant $c$:
$$
S_1 \leq
2\exp\biggl(-c(n_r+1)\eta_{n}(r)\times \min\Bigl(\frac{1}{\max (\lambda_i)},
\frac{(n_r+1)\eta_n(r)}{\sum\lambda_i^2}\Bigr)\biggr).
$$
Next, since $D\T D$ and $DD\T$ have the same non zero eigenvalues, $$
\max_{i=1,\dotsc,d_n}\lambda_i
%\leq \max_{0\leq k\leq n_r+1}\sum_{\ell=0}^{n_r+1}\abs{\esp\bigl(D_{r,k}^{(u)}\,X-\esp(D_{r,k}^{(u)}\,X)\bigr)
%\bigl(D_{r,\ell}^{(u)}\,X- \esp(D_{r,\ell}^{(u)}\,X)\bigr)}\\
 \le \max_{0\leq k \leq n_r}\sum_{\ell=0}^{n_r}\abs{\, \mathds{C}_r(k,\ell)}
$$ and
$
\sum\limits_{i=1}^{d_n} \lambda_i^2 = \sum\limits_{i=1}^{d_n} \sum\limits_{j=1}^{d_n} c_{ij}c_{ji} = \sum\limits_{k=0}^{n_r}\sum\limits_{\ell=0}^{n_r} (\sum\limits_{i=1}^{d_n} d_{ki}d_{li})^2
= \sum\limits_{k=0}^{n_r}\sum\limits_{\ell=0}^{n_r}  \mathds{C}_r^2(k,\ell)$. Finally $S_1$ is bounded by
$$
2\exp\biggl(-c (n_r+1)\eta_{n}(r)\times  \min\Bigl( \big(\max\limits_{0\leq k \leq
n_r}\sum\limits_{\ell=0}^{n_r}\abs{\, \mathds{C}_r(k,\ell)}\big)^{-1}, \frac{(n_r +1)\eta_n(r)}{\sum\limits_{k=0}^{n_r}\sum\limits_{\ell=0}^{n_r}  \mathds{C}_r^2(k,\ell)}\Bigr)\biggr).
$$
For $S_2$, we use the  classical exponential bound on a Gaussian variable: $Y \sim {\cal N}(0,\sigma^2)$ implies that
$\p(\abs{Y}\ge \ve) \le \min(1, \sqrt{\frac{2\sigma^2}{\pi\ve^2}}) \exp ( - \frac{\ve^2}{2\sigma^2}), \; \ve >0.$ Here $Y = \sum\limits_{k=0}^{n_r} (\esp D_{r,k}^{(u)} X) ( D_{r,k}^{(u)} X - \esp D_{r,k}^{(u)} X)$ and we get easily that $\var(Y) \le v_n(r)$.
\end{proof}

\subsection{Proofs of section~\ref{Asymptres}}

\begin{proof}\textbf{Proof of Theorem~\ref{t31}}\\
Recall that $\widehat{\ro}$ is given by: $
\widehat{\ro} =\min\Big\{r\in\{2,\dotsc,m_n\}\;\;\;:\;\;\; B_n(r) \text{ holds} \Big\}-2
$ where the event $B_n(r)$ is defined by
$\dsp B_n(r) = \big\{ \overline{
\big( D^{(1)}_r X\big)^2} \ge n^2 b_n\big\}$, and $\widehat{\ro} = \ell_0$ if $\cap_{r=2}^{m_n} B_n(r)$. The condition $m_n \to\infty$ guarantees  that for $n$ large enough, $\ro+2\in\{2,\dotsc,m_n\}$. From this definition, we write
$$
\esp\big(\widehat{\ro} - \ro)^2 = \sum_{r=0}^{m_n -2} (r -\ro)^2\p\big(\widehat{\ro} = r\big) + (l_0 - \ro)^2 \p\big( \widehat{\ro}=l_0\big)
$$
where $\p\big(\widehat{\ro} = 0\big)  = \p\big(B_n(2) \big)$, $
\p\big(\widehat{\ro} = r\big) = \p\big(B_n^c(2)\cap \dotsb \cap B_n^c(r+1) \cap B_n(r+2) \big)$ if $r=1,\dotsc,m_n-2$,  and
$\p\big(\widehat{\ro} = l_0\big)  \le  \p\big( B_n^c(\ro+2)\big)$. Then, for all $\ro\in\n_0$:
$
\esp\big(\widehat{\ro} -\ro \big)^2 = {\mathcal O} \big( T_{1n}(\ro)\big) + {\mathcal O} \big(m_n^3 T_{2n}(\ro)\big)
$
where we have set $T_{1n}(0) = 0$,
$T_{1n}(\ro) = \sum\limits_{r=2}^{\ro+1} \p\big(B_n(r)\big) \text{ (for } \ro\ge 1)$ and $ T_{2n}(\ro) = \p\big( B_n^c(\ro+2) \big)$. Now, the study of $T_{1n}$ and $T_{2n}$ is derived from results of Lemma~\ref{l62}, Lemma~\ref{l63}, Proposition~\ref{p61} and Lemma~\ref{l64}. In particular, since $\mu\in C^{\ro+1}([0,T])$ we get:
$$
\esp (D_{r,k}^{(u)} X) = \sum_{i=0}^r b_{ikr}{(u)} (t_{k+iu}-t_k)^{\ra} \int_0^1 \frac{(1-v)^{\ra-1}}{(\ra - 1)!} \mu^{(\ra)} (t_k + (t_{k+iu} - t_k)v) \dv
$$
which is  ${\cal O} (n^{r -\ra})$ for $\ra=\min(r,\ro+1)$ implying that $\esp (D_{r,k}^{(u)} X)= {\cal O}(1)$ for $r=1,\dotsc,\ro+1$, $\esp (D_{r,k}^{(u)} X)= {\cal O}(n)$ for $r=\ro+2$. Then one may bound $v_n(r)$ given in equation \eqref{e616} by  ${\cal O}(n^2)$  if $r=1,\dotsc,\ro$,
${\cal O} \big(n^{3-2\bo}\indi_{]0,\frac{1}{2}[}(\bo) + n^2\ln n\indi_{\{\frac{1}{2}\}}(\bo) + n^2\indi_{]\frac{1}{2},1[}(\bo)\big)$   if  $r=\ro+1$  with   A\ref{h21}-(iii-1), ${\cal O} \big(n^{7-2\bo}\indi_{]0,\frac{1}{2}[}(\bo) + n^6\ln n\indi_{\{\frac{1}{2}\}}(\bo) + n^6\indi_{]\frac{1}{2},1[}(\bo)\big)$  if  $r=\ro+2$ with  A\ref{h21}-(iii-1), and ${\cal O} (n^{7-2\bo})$ if  $r=\ro+2$ and  A\ref{h21}-(iii-2) holds. Next after some calculations based on properties $n^{2\bo}b_n\to \infty$ and $n^{-2(1-\bo)}b_n\to 0$, one may derive from Proposition \ref{p62}  that:
$$
T_{1n}(\ro) =  {\mathcal O} \Big(\exp\Big(- D(\ro) b_n \big(n^{2\bo+1}\indi_{]0,\frac{1}{2}[}(\bo) + \big(\frac{n^2}{\ln n}\big) \indi_{\{\frac{1}{2}\}}(\bo) + n^2\indi_{]\frac{1}{2},1[}(\bo)\big)\Big) \Big).$$
Next, if  A\ref{h21}-(iii-1) holds
$$T_{2n}(\ro) = O \bigg( \exp\Big( - D(\ro) \big(n\indi_{]0,\frac{1}{2}[}(\bo) + \big(\frac{n}{\ln n}\big)\indi_{\{\frac{1}{2}\}}(\bo) + n^{2(1-\bo)}\indi_{]\frac{1}{2},1[}(\bo) \big)\Big) \bigg)$$
while, under A\ref{h21}-(iii-2) and for all $\bo\in]0,1[$, $T_{2n}(\ro) =O \big( \exp( - D(\ro) n) \big)$. For $p=1,2$, we get that $T_{1n}(\ro) = o\big( T_{2n} (\ro)\big)$ and the mean square error follows. Finally, to obtain a bound for $\p(\widehat{\ro} \not=\ro)$, it suffices to notice that $\{\widehat{\ro} = 0\} = B_n(2)$ for $\ro=0$ and $\{\widehat{\ro} = \ro\} = B_n^c(2) \cap \dotsb \cap B_n^c(\ro+1) \cap B_n(\ro+2)$ for $\ro \ge 1$, by this way
$\p(\widehat{\ro} \not= \ro) = T_{1n}(\ro) + T_{2n}(\ro)= T_{2n}(\ro)(1+o(1))$.
\end{proof}

\begin{proof}\textbf{Proof of Theorem~\ref{t32}}\\
We start the proof, with either $p=1$ or $p=2$, and thus denote by $\widehat{r_{\!p}}$  (resp. $r_{\!p}$) the
quantity $\widehat{\ro}+p$ (resp. $\ro+p$). We set
\begin{multline}\label{e617}
l_n(p,\ro,\bo)= -\frac{1}{2n} \sum_{k=0}^{n}d_0(t_k)\psi^{2(p-\bo)}(t_k) \sum_{i,j=0}^{r_{\!p}} \frac{(ij)^{\ro}}{\prod_{\substack{m=0\\m\not=i}}^{r_{\!p}} (i-m)\prod_{\substack{q=0\\q\not=j}}^{r_{\!p}} (j-q)}
 \\  \times \intd_{[0,1]^2} \frac{((1-v)(1-w))^{\ro -1}}{(\ro - 1)!^2}\abs{iv-jw}^{2\bo}  \dv\!\dw,
\end{multline}
for all $\ro \ge 1$ while if $\ro=0$,
\begin{equation} \label{e618}
l_n(p,0,\bo) = -\frac{1}{2n} \sum_{k=0}^{n} d_0(t_k)\psi^{2(p-\bo)}(t_k)\sum_{i,j=0}^{r_{\!p}}
\frac{ \abs{i-j}^{2\bo}}{\prod_{\substack{m=0\\m\not=i}}^{p} (i-m)\prod_{\substack{q=0\\q\not=j}}^{p} (j-q)}.
\end{equation}
We study the convergence of $\widehat{\alpha}_p = 2(\widehat{\beta}_n^{(p)} -p)$ toward $\alpha_p = 2(\bo -p)$, so that
$$ \widehat{\alpha}_p = \frac{\ln \big(  \overline{(D_{\widehat{r_{\!p}}}^{(u)} X)^2}\big) -  \ln\big(\overline{(D_{\widehat{r_{\!p}}}^{(v)} X)^2}\big)}{\ln(u/v)}.$$
We consider the following decomposition of $\ln(u/v)\widehat{\alpha}_p$:
\begin{multline*}
\ln\big(\frac{n^{\alpha_p}}{n-u\widehat{r}_p+1} \sum_{k=0}^{n-u\widehat{r}_p}
\big( D_{\widehat{r_{\!p}},k}^{(u)}\,X\big)^2-u^{\alpha_p}l_n(p,\ro,\bo)+u^{\alpha_p}l_n(p,\ro,\bo)\big)\\
-\ln\big( \frac{n^{\alpha_p}}{n-v\widehat{r_{\!p}}+1} \sum_{k=0}^{n-v\widehat{r_{\!p}}} \big(
D_{\widehat{r_{\!p}},k}^{(v)}\,X\big)^2-v^{\alpha_p}l_n(p,\ro,\bo)+v^{\alpha_p}l_n(p,\ro,\bo)\big)
%\\
%=H\ln(u/v)+\frac{\frac{n^{H}}{n_r+1} \sum_{k=0}^{n_r} \big( D_{r,k}^{(u)}\,X\big)^2-u^{H}\ell}{u^{H}\ell}+\frac{\frac{n^{H}}{n-vr+1} \sum_{k=0}^{n-vr} \big( \Delta_r^{(v)}X_k\big)^2-v^{H}\ell}{v^{H}\ell}+o()
\end{multline*}

\noindent Hence $\ln(u/v)(\widehat{\alpha}_p-\alpha_p)=F_n(u)-F_n(v)+o(F_n(u)+F_n(v))$
where $o(\cdot)\tv[n\to\infty]{a.s.} 0$ as soon as $F_n(\cdot)\tv[n\to\infty]{a.s.} 0$ with
$$F_n(u)=\frac{n^{\alpha_p} \overline{\big(
D_{\widehat{r_{\!p}}}^{(u)}X\big)^2}-u^{\alpha_p}l_n(p,\ro,\bo)}{u^{\alpha_p}l_n(p,\ro,\bo)}=
\frac{F_{1,n,p}(u)+F_{2,n,p}(u)+F_{3,n,p}(u)}{u^{\alpha_p}l_n(p,\ro,\bo)}
$$
for $\dsp  F_{1,n,p}(u) = n^{\alpha_p} \Big(\overline{\big( D_{\widehat{r_{\!p}}}^{(u)}X\big)^2}-\overline{ \big( D_{r_{\!p}}^{(u )}X\big)^2}\Big)$, $\dsp F_{2,n,p}(u) = n^{\alpha_p}\big( \overline{ \big( D_{r_{\!p}}^{(u)} X\big)^2}- \esp\overline{\big( D_{r_{\!p}}^{(u)}X\big)^2} \big)$ and
$\dsp  F_{3,n,p}(u) = n^{\alpha_p} \esp \Big(\,\overline{\big( D_{r_{\!p}}^{(u)}X\big)^2}\,\Big) - u^{\alpha_p} \, l_n(p,\ro,\bo)$.

\paragraph{{\bf (i) \emph{Study of}} $\dsp F_{1,n,p}(u)$.}

From Theorem~\ref{t31}, we get that $\sum\limits_n\p({\widehat\ro}\not= \ro ) < \infty$, so, a.s. for $n$ large enough, $\widehat{\ro} = \ro$  and $F_{1,n,p}(u)\equiv 0$, $p=1$ or $p=2$.

\paragraph{{\bf (ii) \emph{Study of $F_{2,n,p}(u)$.}}}

We  study $$\p \big(  \abs{ \overline{ \big( D_{r_{\!p}}^{(u)} X\big)^2}- \esp\overline{\big( D_{r_{\!p}}^{(u)}X\big)^2}   } > c_p n^{2(p-\bo)} \psi_{np}^{-1}(\bo)\big)$$ for $c_p$ a positive constant, $\psi_{n2}(\bo) \equiv \big(\frac{n}{\ln n}\big)^{\frac{1}{2}}$ and
$$
\psi_{n1}(\bo) = \big(\frac{
n}{\ln n}\big)^{\frac{1}{2}} \indi_{]0,\frac{3}{4}[}(\bo) +\big(\frac{n^{1/2}}{\ln n}\big)\indi_{\{\frac{3}{4}\}}(\bo)+
 \big(\frac{n^{2(1-\bo)}}{\ln n}\big)\indi_{]\frac{3}{4},1[}(\bo).
$$
We apply Lemma~\ref{l64} and Proposition~\ref{p62} with $p=1$ or $p=2$. After some calculations and the application of Borel Cantelli's lemma with $c_p$ chosen large enough, we  obtain that for $p=1$, almost surely,
$
\varlimsup\limits_{n\to\infty} \psi_{np}(\bo) \abs{F_{2,n,p}(u)} < +\infty
$ under the condition~ A\ref{h21}-(iii-$p$), where $p=1$ or $2$.

\paragraph{\bf (iii)  \emph{Study of} $F_{3,n,p}(u)$.}

From  \eqref{e611} and proceeding similarly as in \eqref{e612}, we get for $\ro \ge 1$, that $n^{\beta_1} \big(n^{\alpha_p}\esp\overline{\big(D^{(u)}_{r_{\!p}} X \big)^2} - u^{\alpha_p} \, l_n(p,\ro,\bo) \big)$ could be decomposed into $B_{n1}+ B_{n2} + B_{n3}$ with
\begin{multline*}
B_{n1}=  -\frac{n^{\alpha_p + \bun}}{2(n-ur_p+1)} \sum_{k=0}^{n-ur_p} \sum_{i,j=0}^{r_{\!p}} b_{ikr}^{(u)} b_{jkr}^{(u)} (t_{k+iu} - t_k)^{\ro} (t_{k+ju} - t_k)^{\ro} \\ \times \intd_{[0,1]^2} \frac{((1-v)(1-w))^{\ro -1}}{(\ro - 1)!^2}   \abs{\dot{v}_{ik} - \dot{w}_{jk}}^{2\bo +\bun} \\ \times
\bigg\{ \frac{\frac{\ll^{(\ro,\ro)}(\dot{v}_{ik},\dot{w}_{jk}) -\frac12 \ll^{(\ro,\ro)}(\dot{v}_{ik},\dot{v}_{ik}) -\frac12
\ll^{(\ro,\ro)}(\dot{w}_{jk},\dot{w}_{jk})}{\abs{\dot{v}_{ik} - \dot{w}_{jk}}^{2\bo}}-d_0(\dot{w}_{jk})}{\abs{\dot{v}_{ik} - \dot{w}_{jk}}^{\bun}} - d_1(\dot{w}_{jk}) \bigg\} \dv\dw
\end{multline*}
\begin{multline*}
B_{n2} =  -\frac{n^{\alpha_p + \bun}}{2(n-ur_p+1)} \sum_{k=0}^{n-ur_p} \sum_{i,j=0}^{r_{\!p}} b_{ikr}^{(u)} b_{jkr}^{(u)} (t_{k+iu} - t_k)^{\ro} (t_{k+ju} - t_k)^{\ro} \\ \times \intd_{[0,1]^2} \frac{((1-v)(1-w))^{\ro -1}}{(\ro - 1)!^2}   \abs{\dot{v}_{ik} - \dot{w}_{jk}}^{2\bo+\bun} d_1(\dot{w}_{jk}),
\end{multline*}
\begin{multline*}
B_{n3}=  n^{\bun}\Big( \frac{-n^{\alpha_p}}{2(n-ur_p+1)} \sum_{k=0}^{n-ur_p} \sum_{i,j=0}^{r_{\!p}} b_{ikr}^{(u)} b_{jkr}^{(u)} (t_{k+iu} - t_k)^{\ro} (t_{k+ju} - t_k)^{\ro} \\ \times \intd_{[0,1]^2} \frac{((1-v)(1-w))^{\ro -1}}{(\ro - 1)!^2}   \abs{\dot{v}_{ik} - \dot{w}_{jk}}^{2\bo }  d_0(\dot{w}_{jk}) \dv\dw - u^{\alpha_p} l_n(p,\ro,\bo)\Big)
%\\+ \frac{u^{\alpha_p}}{2n} \sum_{k=0}^n \psi^{\alpha_p}(t_k)d_0(t_k)  \sum_{i,j=0}^{r_{\!p}} \frac{(ij)^{\ro}}{\prod_{\substack{m=0\\m\not=i}}^{r_{\!p}}(i-m)\prod_{\substack{q=0\\q\not=j}}^{r_{\!p}}(j-q)} \\  \times \intd_{[0,1]^2} \frac{((1-v)(1-w))^{\ro -1}}{(\ro - 1)!^2}   \abs{\dot{v}_{ik} - \dot{w}_{jk}}^{2\bo }\dv\dw.
\end{multline*}
with $l_n(p,\ro,\bo)$ given by \eqref{e617}. Next, using Lemma~\ref{l62} and \ref{l63} and  the condition \eqref{e31} with uniform continuity of $d_1(\cdot)$, we get that $B_{n1} = o(1)$ and $B_{n2}$ has the limit:
\begin{multline*}
 -\frac{u^{\alpha_p+\bun}}{2}  \sum_{i,j=0}^{r_{\!p}} \frac{(ij)^{\ro} \int_0^T d_1(t) \psi^{1-\alpha_p -\bun}(t) \dt}{\prod_{\substack{m=0\\m\not=i}}^{r_{\!p}}(i-m)\prod_{\substack{q=0\\q\not=j}}^{r_{\!p}}(j-q)}
\\\times \intd_{[0,1]^2} \frac{\big((1-v)(1-w)\big)^{\ro-1}}{((\ro - 1)!)^2}
\abs{iv -jw}^{2\bo+\bun} \dv\!\dw.
\end{multline*}
For the last term $B_{n3}$, one may show that it is of order ${\cal O} (n^{\bun -1})$. Finally, the case $\ro=0$ is treated similarly from \eqref{e610}.

\paragraph{\bf Conclusion.} One may note that the determinist term, $l_n(p,\ro,\bo)$, defined in \eqref{e617}-\eqref{e618}, converges to the nonzero term:
$$
 -\frac{1}{2} \sum_{i,j=0}^{r_{\!p}}\frac{ (ij)^{\ro}\int_0^T d_0(t) \psi^{-\alpha_p+1}(t)\dt}{\prod\limits_{\substack{m=0\\m\not=i}}^{r_{\!p}}(i-m)\prod\limits_{\substack{q=0\\q\not=j}}^{r_{\!p}}(j-q)}
  \intd_{[0,1]^2}\!\!\!\!\!\!\frac{((1-v)(1-w))^{\ro -1}}{(\ro - 1)!^2}\abs{iv-jw}^{2\bo}\!\!\dv\!\dw$$
for $\ro \ge 1$ while if $\ro=0$, the limit is
$ -\frac{1}{2} \sum_{i,j=0}^{p}
\frac{\abs{i-j}^{2\bo}\int_0^T d_0(t) \psi^{-\alpha_p+1}(t)\dt}{\prod_{\substack{m=0\\m\not=i}}^{p}(i-m)\prod_{\substack{q=0\\q\not=j}}^{p}(j-q)}$.
\end{proof}

\subsection{Proofs of section~\ref{AppInt}}
\begin{proof}\textbf{Proof of Theorem~\ref{t41}}\\
We set $\widetilde{\ro} = \max(\widehat{\ro},1)$ and, for $\widehat{\ro}$ and $\wt{X}_r(\cdot)$  respectively defined in \eqref{e25} and \eqref{e41}, we  use the convention:~ $ \wt{X}_{\widetilde{\ro}}(\cdot)= \wt{X}_{m_n-1}(\cdot)$ and $\wt{X}_{\widehat{\ro}+1}(\cdot)= \wt{X}_{m_n}(\cdot)$ when $\widehat{\ro} = l_0$.

(a) If $\rr = \max(r,1)$ and $\overline{\ro}=\max(\ro,1)$, we get, for $n$  large enough such that $\ro \le m_n -2$,
\begin{align*}
\big(X(t) -  \wt{X}_{\widetilde{\ro}}(t) \big)^2
&=\sum_{r=0}^{m_n -2} \big(X(t) -  \wt{X}_{\rr}(t) \big)^2\indi_{\{\widehat{\ro}=r\}}+ \big(X(t) -  \wt{X}_{m_n-1}(t) \big)^2\indi_{\{\widehat{\ro}=l_0\}}
\\
&\le  \big(X(t) -  \wt{X}_{\overline{\ro}}(t) \big)^2+ \indi_{\{\widehat{\ro}\not=\ro\}}\sum_{r=0,r\not=\ro}^{m_n -1} \big(X(t) -  \wt{X}_{\rr}(t) \big)^2.
\end{align*}
By this way, $e_{\rho}^2 ({\mathrm{app}}\big(\widehat{\ro})\big)$ should be bounded by
$$
\int_0^T \!\!\!\esp \big(X(t) -  \wt{X}_{\overline{\ro}}(t) \big)^2 \rho(t)\dt+ \big(\p(\widehat{\ro} \not=\ro)\big)^{\frac{1}{2}} \!\!\!\sum_{r=0,r\not=\ro}^{m_n -1}\!\!\int_0^T \!\! \Big(\esp\big(X(t) -  \wt{X}_{\rr}(t) \big)^4\Big)^{\frac{1}{2}} \rho(t)\dt
$$
We make use of the exponential bound established for $\p(\widehat{\ro} \not=\ro)$ in Theorem~\ref{t31} as well as the property $\esp(Y^4) \le 3 \big(\esp(Y^2)\big)^2$ for a Gaussian r.v. $Y$. Moreover, $\dsp
\sup_{t\in[0,T]}\Big(\esp\big(X(t) -  \wt{X}_{r}(t)\big)^2\Big) = \max\limits_{k=0,\dotsc,\lfloor \frac{n}{r}\rfloor - 1} \sup\limits_{t\in {\cal I}_k} \Big(\esp\big(X(t) -  \wt{X}_{r}(t)\big)^2\Big)$. If $\ro \ge 1$, we use the decomposition established in \citet[][lemma 4.1]{BV08} to obtain, for $t\in{\cal I}_k$ and $r^{\ast} = \min(r,\ro)$:
\begin{multline*}
\esp\big(X(t) -  \wt{X}_{r}(t)\big)^2 = \sum_{i,j=0}^{r} L_{i,k,r}(t) L_{j,k,r}(t) \frac{(t_{kr+i}-t_{kr})^{\ra}(t_{kr+j}-t_{kr})^{\ra}}{((\ra-1)!)^2} \\\times \intd_{[0,1]^2} \big((1-v)(1-w)\big)^{\ra -1}  \Big\{ \ll^{(\ra,\ra)}(t_{kr}+ (t-t_{kr})v,t_{kr} + (t-t_{kr})w) \\- \ll^{(\ra,\ra)}(t_{kr}+ (t-t_{kr})v,t_{kr} + (t_{kr+j}-t_{kr}) w) \\- \ll^{(\ra,\ra)}(t_{kr}+ (t_{kr+i} - t_{kr}) v, t_{kr} + (t-t_{kr})w)\\ + \ll^{(\ra,\ra)}(t_{kr}+ (t_{kr+i}  -t_{kr}) v,t_{kr} + (t_{kr+j} - t_{kr})w) \Big\} \dv\!\dw.
\end{multline*}
If $r=1,\dotsc,\ro-1$, $(\ro \ge 2)$, we obtain the uniform bound ${\cal O} \big( \dn^{2r+2} \big)$ by uniform  continuity of $\ll^{(r+1,r+1)} (\cdot,\cdot)$ and results of Lemma~\ref{l62}.  For $r=\ro,\dotsc,m_n$, we have $r^{\ast} = \ro$ so we apply the H\"{o}lderian regularity condition \eqref{e69}. Since $L_{i,k,r}(t) \le r^r$, we arrive at $\dsp \sup_{t\in[0,T]}\esp\big(X(t) -  \wt{X}_{\overline{\ro}}(t)\big)^2  = {\cal O}\big(\dn^{2(\ro+\bo)}\big)$ for $r=\ro$ while if  $r=\ro+1,\dotsc,m_n$, $\dsp \sup_{t\in[0,T]}\esp\big(X(t) -  \wt{X}_{\rr}(t)\big)^2  = {\cal O}\big(m_n^{2(m_n+\ro+\bo)} \dn^{2(\ro+\bo)}\big)$. The logarithmic order of $m_n$ yields the final result. In the case where $\ro = 0$, above results hold true starting from \begin{multline*}
 \esp\big(X(t) -  \wt{X}_{\rr}(t)\big)^2 = \sum_{i,j=0}^{\rr} L_{i,k,\rr}(t) L_{j,k,\rr}(t)\Big\{ \ll(t,t) - \ll(t,t_{k\rr+j})  \\- \ll(t_{k\rr+i},t) + \ll(t_{k\rr+i},t_{k\rr+j})   \Big\}.\end{multline*}

(b) For $e_{\rho}^2 ({\rm{int}}\big(\widehat{\ro})\big)$,  $\int_0^T \big( X(t) -\widetilde{X}_{r+1}\big)\rho(t)\dt$ is again a Gaussian variable, so in a similar way as for approximation, we get the following bound for this term:
\begin{multline*}
\sqrt{3}\big(\p(\widehat{\ro} \not=\ro)\big)^{\frac{1}{2}} \sum_{r=0}^{m_n}
\Big(\sup_{t\in[0,T]}\Big(\esp\big(X(t) -  \wt{X}_{r+1}(t)\big)^2\Big)^{\frac{1}{2}}\Big)^2 \big(\int_0^T  \rho(t)\dt\big)^2
\\+
 \sum_{k=0}^{\lfloor \frac{n}{\ro+1}\rfloor-1}\sum_{\ell=0}^{\lfloor \frac{n}{\ro+1}\rfloor-1} \int_{{\cal I}_k}\int_{{\cal I}_{\ell}}\esp \big(X(t) -  \wt{X}_{\ro+1}(t) \big) \big(X(s) -  \wt{X}_{\ro+1}(s) \big)\rho(t)\rho(s)\,\mathrm{d}s\!\dt.
\end{multline*}

\paragraph{Study of the term $\dsp\esp \big(X(t) -  \wt{X}_{\ro+1}(t) \big) \big(X(s) -  \wt{X}_{\ro+1}(s) \big) $, $(s,t)\in {\cal I}_{\ell}\times {\cal I}_k$.}

Denoting $\overline{r}=\ro+1$ we get again from lemma 4.1 of \citet{BV08} that $\esp \big(X(t) -  \wt{X}_{\overline{r}}(t) \big) \big(X(s) -  \wt{X}_{\overline{r}}(s) \big)$ is equal to:
{\small \begin{multline*}
\sum_{i,j=0}^{\overline{r}} L_{i,k,\overline{r}}(t)  L_{j,\ell,\overline{r}}(s)   \frac{((t_{k \overline{r} +i}-t_{k \overline{r} })(t_{\ell \overline{r} +j}-t_{\ell \overline{r} }))^{\ro}}{((\ro-1)!)^2} \intd_{[0,1]^2}\!\!\! \!\!\!\dv\!\dw\, ((1-v)(1-w))^{\ro-1}\\ \times
\Big\{
\ll^{(\ro,\ro)} (t_{k \overline{r} }+(t-t_{k \overline{r} })v, t_{\ell \overline{r} }+(t-t_{\ell \overline{r} })w)
-\ll^{(\ro,\ro)} (t_{k \overline{r} }+(t-t_{k \overline{r} })v, t_{\ell \overline{r} }+ (t_{\ell \overline{r} +j} - t_{\ell \overline{r} } )w)
\\- \ll^{(\ro,\ro)} (t_{k \overline{r} }+ (t_{k \overline{r} +i} - t_{k \overline{r} } )v, t_{\ell \overline{r} }+(t-t_{\ell \overline{r} })w)
+ \ll^{(\ro,\ro)} (t_{k \overline{r} }+ (t_{k \overline{r} +i} - t_{k \overline{r} } )v, t_{\ell \overline{r} }+(t-t_{\ell \overline{r} })w)\Big\}.
\end{multline*}}
For non-overlapping intervals ${\cal I}_k$ and ${\cal I}_{\ell}$, that is $\abs{k-l}\ge 2$, we make use of Condition A\ref{h22}(2) four times, by adding and subtracting the necessary terms,  noting that
\begin{multline*}
\sum_{i,j=0}^{ \overline{r} } L_{i,k, \overline{r} }(t)L_{j,\ell, \overline{r} }(s) (t_{k \overline{r} +i} - t_{k \overline{r} })^{r_1}(t_{\ell \overline{r} +j} - t_{\ell  \overline{r} })^{r_2}= (t-t_{k \overline{r} })^{ r_1}(s- t_{\ell  \overline{r} })^{ r_2}.
\end{multline*}
with either $r_i= \overline{r} -1$ or $r_i = \overline{r} $ for $i=1,2$. By this way, we get
\begin{multline*}
\sum_{\stackrel{k,\ell=0}{\abs{k-\ell}\ge 2}}^{\lfloor \frac{n}{\overline{r}}\rfloor-1} \int_{{\cal I}_k}\int_{{\cal I}_{\ell}}\esp \big(X(t) -  \wt{X}_{\overline{r}}(t) \big) \big(X(s) -  \wt{X}_{\overline{r}}(s) \big)\rho(t)\rho(s)\,\mathrm{d}s\!\dt\\ = {\cal O}\Big(\dn^{2(\ro+\bo+1)}\sum_{\stackrel{k,\ell=0}{\abs{k-\ell}\ge 2}}^{\lfloor \frac{n}{\overline{r}}\rfloor-1} \big\lvert\abs{k-\ell}-1\big\rvert^{-2(2-\bo)}  \Big)
 %\\= {\cal O}\Big(n\dn^{2(\ro+\bo+1)}\Big) = {\cal O}\Big(\dn^{2(\ro+\bo)+1}\Big).
\end{multline*}
which is a ${\cal O}\Big(\dn^{2(\ro+\bo)+1}\Big)$. For overlapping intervals ${\cal I}_k$ and ${\cal I}_{\ell}$, that is in the case where $\abs{k-l} \le 1$, we make use of Cauchy-Schwarz inequality to obtain the same bound as above. Since the second part of  $e_{\rho}^2 (\rm{int}\big(\widehat{\ro})\big) $ is  negligible, we obtain the  result.
\end{proof}

\bibliographystyle{chicago}
\bibliography{bv2014}
\end{document}